\documentclass[11pt,reqno]{amsart}
\usepackage{fourier}
\usepackage{fullpage}
\usepackage{mathrsfs,amssymb,graphicx,verbatim,amsmath,amsfonts}
\usepackage{paralist}
\usepackage[breaklinks,pdfstartview=FitH]{hyperref}
\usepackage{upgreek}
\usepackage[mathscr]{euscript}
\makeatletter
\def\resetMathstrut@{%
  \setbox\z@\hbox{%
    \mathchardef\@tempa\mathcode`\(\relax
    \def\@tempb##1"##2##3{\the\textfont"##3\char"}%
    \expandafter\@tempb\meaning\@tempa \relax
  }%
  \ht\Mathstrutbox@1.2\ht\z@ \dp\Mathstrutbox@1.2\dp\z@
}
\makeatother
\renewcommand{\le}{\leqslant}
\renewcommand{\ge}{\geqslant}

\renewcommand{\setminus}{\smallsetminus}
\renewcommand{\gamma}{\upgamma}
\newcommand{\ud}[0]{\,\mathrm{d}}

\usepackage{dutchcal}

\newcommand{\imb}{\text{\tt Imbalance}}

\newcommand{\bd}{{\boldsymbol\delta}}
\newcommand{\cc}{\mathsf{c}}
\newcommand{\ZZ}{\mathbf{Z}}
\newcommand{\NN}{\mathcal{N}}
\newcommand{\opt}{\mathrm{opt}}

\newcommand{\proj}{\mathsf{Proj}}
\newcommand{\sign}{\mathrm{sign}}

\newcommand{\n}{\{1,\ldots,n\}}

\newcommand{\cH}{\mathcal{H}}
\newcommand{\cL}{\mathcal{L}}
\newcommand{\X}{\mathbf X}
\newcommand{\Y}{\mathbf Y}
\newcommand{\f}{\varphi}

\renewcommand{\d}{\delta}

\newcommand{\e}{\varepsilon}
\newcommand{\R}{\mathbb R}
\newcommand{\1}{\mathbf 1}

\newtheorem{theorem}{Theorem}
\newtheorem{lemma}[theorem]{Lemma}
\newtheorem{proposition}[theorem]{Proposition}

\newtheorem{corollary}[theorem]{Corollary}

\newtheorem{definition}[theorem]{Definition}

\theoremstyle{remark}
\newtheorem{remark}[theorem]{Remark}
\newcommand{\MM}{\mathcal{M}}

\newtheorem{question}[theorem]{Question}

\newcommand{\EE}{\mathbf{E}}

\renewcommand{\subset}{\subseteq}
\newcommand{\E}{\mathbb{ E}}

\newcommand{\W}{\mathsf{W}}

\newcommand{\N}{\mathbb N}

\newcommand{\eqdef}{\stackrel{\mathrm{def}}{=}}

\newcommand{\Lip}{\mathrm{Lip}}
\newcommand{\BM}{\mathrm{BM}}

\newcommand{\ee}{\mathsf{e}}

\begin{document}

\title{Impossibility of almost extension}

%\dedicatory{Dedicated with awe to the memory of Jean Bourgain}

\author{Assaf Naor}
%\address{Mathematics Department\\ Princeton University}
\address{Mathematics Department\\ Princeton University\\ Fine Hall, Washington Road, Princeton, NJ 08544-1000, USA}
\email{naor@math.princeton.edu}
\thanks{Supported in part by the BSF, the Packard Foundation and the Simons Foundation. The research that is presented here was conducted under the auspices of the Simons Algorithms and Geometry (A\&G) Think Tank.}

%\date{\today}
%\keywords{Lipschitz extension}
\maketitle

%\vspace{-0.33in}
%\allowdisplaybreaks
\begin{abstract}
\begin{comment}
Bourgain's almost extension theorem (1987) states that if $(\X,\|\cdot\|_\X), (\Y,\|\cdot\|_\Y)$ are normed spaces with $\dim(\X)<\infty$, then for any $\e>0$, any $\e$-net $\NN$ of the unit sphere of $\X$, and any $1$-Lipschitz mapping $f:\NN\to \Y$, there exists an $O(1)$-Lipschitz  mapping  $F:\X\to \Y$ such that  $\|F(a)-f(a)\|_\Y\lesssim  \dim(\X)\e$ for all $a\in \mathcal{N}$. Here, we prove  that this theorem is optimal up to lower order factors (as $\dim(\X)\to \infty$), i.e., sometimes  $\max_{a\in \NN} \|F(a)-f(a)\|_\Y\gtrsim \dim(\X)^{1-o(1)}\e$ for {\em every} $O(1)$-Lipschitz mapping  $F:\X\to \Y$. This improves over the lower bound (also obtained by Bourgain, in the same 1987 work) of $\max_{a\in \NN} \|F(a)-f(a)\|_\Y\gtrsim \dim(\X)^{c}\e$ for some (unspecified) universal constant $0<c<\frac12$. When $\X=\mathbf{H}$ is a Hilbert space, the approximation in the almost extension theorem can be improved to $\max_{a\in \NN} \|F(a)-f(a)\|_\Y\lesssim  \sqrt{\dim(\mathbf{H})}\e$, as shown in our forthcoming work with Schechtman. Here, we prove that this Hilbertian almost extension theorem is optimal, i.e.,  sometimes  $\max_{a\in \NN} \|F(a)-f(a)\|_\Y\gtrsim \sqrt{\dim(\mathbf{H})}\e$ for {\em every} $O(1)$-Lipschitz mapping $F:\mathbf{H}\to \Y$.
\end{comment}

Let  $(\X,\|\cdot\|_\X), (\Y,\|\cdot\|_\Y)$ be normed spaces with $\dim(\X)=n$. Bourgain's almost extension theorem  asserts that for any $\e>0$, if  $\NN$ is an $\e$-net of the unit sphere of $\X$ and $f:\NN\to \Y$ is $1$-Lipschitz, then there exists an $O(1)$-Lipschitz   $F:\X\to \Y$ such that  $\|F(a)-f(a)\|_\Y\lesssim  n\e$ for all $a\in \mathcal{N}$. We prove  that this is optimal up to lower order factors, i.e., sometimes  $\max_{a\in \NN} \|F(a)-f(a)\|_\Y\gtrsim n^{1-o(1)}\e$ for {\em every} $O(1)$-Lipschitz  $F:\X\to \Y$. This improves  Bourgain's lower bound of $\max_{a\in \NN} \|F(a)-f(a)\|_\Y\gtrsim n^{c}\e$ for some  $0<c<\frac12$. If $\X=\ell_2^n$, then the approximation in the almost extension theorem can be improved to $\max_{a\in \NN} \|F(a)-f(a)\|_\Y\lesssim  \sqrt{n}\e$. We prove that this is sharp, i.e.,  sometimes  $\max_{a\in \NN} \|F(a)-f(a)\|_\Y\gtrsim \sqrt{n}\e$ for {\em every} $O(1)$-Lipschitz $F:\ell_2^n\to \Y$.

\end{abstract}

\section{Introduction}

The following theorem was proved by Bourgain in~\cite{Bou87}; see~\cite{Beg99} for a simpler proof.
\begin{theorem}[Bourgain's almost extension theorem~\cite{Bou87}]\label{thm:bourgain} There exists a universal constant $C\ge 1$ with the following property. Fix $n\in \N$, $\e>0$ and  $L> C$. Suppose that $(\X,\|\cdot\|_\X)$ is an $n$-dimensional normed space and that $\mathcal{N}$ is an $\e$-net of the unit sphere of $\X$. Let $(\Y,\|\cdot\|_\Y)$ be a Banach space, and let $f:\mathcal{N}\to \Y$ be a $1$-Lipschitz mapping. Then, there exists an $L$-Lipscitz mapping $F:\X\to \Y$ satisfying\footnote{In addition to the usual $o(\cdot),O(\cdot)$ notation, we will also use  the following (standard) asymptotic notation. For $P,Q>0$, the notations
$P\lesssim Q$ and $Q\gtrsim P$  mean that $P\le KQ$ for a
universal constant $K>0$. The notation $P\asymp Q$
stands for $(P\lesssim Q) \wedge  (Q\lesssim P)$.} \begin{equation}\label{eq:bourgain conclusion}\max_{a\in \NN}\|F(a)-f(a)\|_\Y\lesssim \frac{n}{L}\e.\end{equation} %for every $a\in \NN$.
\end{theorem}
\begin{comment}
For Theorem~\ref{thm:bourgain} and the ensuing discussion, recall the following standard terminology. Given a metric space $(\MM,d_\MM)$  and $\e,\d>0$, a subset $\NN\subset \MM$ is called an $(\e,\d)$-net of $\MM$ if it is  $\e$-separated and $\d$-dense; the former means that  $d_\MM(a,b)\ge \e$ for every distinct $a,b\in \NN$, and the latter means that for every $x\in \MM$ there exists $a\in \NN$ with $d_\MM(x,a)\le \d$. An $\e$-net is the same as an $(\e,\e)$-net.
\end{comment}

\begin{comment}
\begin{remark} It seems likely that there is a version of Theorem~\ref{thm:bourgain} for $C=1$, i.e., that for every $L>1$ one can ensure that   $F$ is $L$-Lipschitz on the unit sphere of $\X$, and that  $\|F(a)-f(a)\|_\Y\lesssim \frac{n}{L-1}\e$ for every $a\in \NN$ (so, the approximation deteriorates as $L\to 1^+$). Lemma~2.1 in~\cite{GNS12} comes close to such a statement, but not quite. Notwithstanding this potential almost-isometric refinement, the main  contribution (and dramatic improvement over the classical Lipschitz extension problem) of Theorem~\ref{thm:bourgain} is when $L$ is any universal constant.
\end{remark}
\end{comment}

It follows from~\cite{Bou87} (specifically, see the reasoning  immediately after the statement of Theorem~2 in~\cite{Bou87}) that in the setting of Theorem~\ref{thm:bourgain}  we must sometimes have $\max_{a\in \NN} \|F(a)-f(a)\|_\Y\ge n^c\e$ for {\em every} $O(1)$-Lipschitz  function $F:\X\to \Y$, where  $0<c<\frac12$ is a universal constant; we will recall the  approach of~\cite{Bou87} and explain its features and limitations in Section~\ref{sec:ext intro}. Theorem~\ref{thm:l1} improves this lower bound via a route that is entirely different from that of~\cite{Bou87}; it demonstrates that Theorem~\ref{thm:bourgain} is sharp up to lower order factors, e.g.~in the important regime $L=O(1)$ it shows that the conclusion of Theorem~\ref{thm:bourgain} cannot be improved to $\max_{a\in \NN}\|F(a)-f(a)\|_\Y\lesssim n\exp\left(-K\sqrt{\log n}\right)\e$ for some universal constant $K>0$.

In the statement of Theorem~\ref{thm:l1}, and throughout what follows, the (closed) unit ball and unit sphere of a normed space $(\X,\|\cdot\|_\X)$ will be denoted $B_\X\subset \X$ and $S_\X\subset \X$, respectively. Namely,
$$
B_\X\eqdef\big\{x\in \X:\ \|x\|_\X\le 1\big\}\qquad\mathrm{and}\qquad S_\X\eqdef \big\{x\in \X:\ \|x\|_\X=1\big\}.
$$
For every $n\in \N$, the unit sphere of the Euclidean $n$-space $\ell_2^n$  will be denoted as usual $S^{n-1}=S_{\ell_2^n}$.

\begin{theorem}\label{thm:l1} Fix $n\in \N$ and $L\ge 1$ with $5L\le\sqrt[4]{n}$. There are Banach spaces $(\X,\|\cdot\|_\X),(\Y,\|\cdot\|_\Y)$ with $\dim(\X)=n$ and $\e>0$ such that for any    $\e$-net $\mathcal{N}$ of $S_{\X}$ there is  a $1$-Lipschitz mapping $f:\mathcal{N}\to \Y$ that satisfies
\begin{equation}\label{eq:max on N}
\max_{a\in \mathcal{N}} \|F(a)-f(a)\|_\Y\gtrsim \frac{Ln}{e^{2\sqrt{(\log n)\log(4L)}}}\e,
\end{equation}
for {\em every} $L$-Lipschitz mapping $F:S_{\X}\to \Y$.
\end{theorem}
Our proof of Theorem~\ref{thm:l1} shows that one could take $\X=\ell_1^n$ and $\Y=\ell_q$, where $q=1+1/ \sqrt{\log_{4L}n}$. While we know how to slightly improve the lower order factor in the right hand side of~\eqref{eq:max on N}, we do not see how to remove it altogether; we suspect that it is possible to do so, but this remains an interesting open question.

The  forthcoming work~\cite{NS18-extension} obtains a refinement of Theorem~\ref{thm:bourgain} that provides  asymptotic improvements (as $n\to \infty$) of~\eqref{eq:bourgain conclusion} under further restrictions on  the source space $(\X,\|\cdot\|_\X)$; in the particularly significant Euclidean special case $\X=\ell_2^n$, the improvement of ~\cite{NS18-extension} is the following theorem.

\begin{theorem}[almost extension from Euclidean nets~\cite{NS18-extension}]\label{thm:NS} There is a universal constant $C\ge 1$ with the following property. Fix $n\in \N$, $\e>0$ and  $L> C$. Suppose  that $\mathcal{N}\subset S^{n-1}$ is an $\e$-net of $S^{n-1}$. Let $(\Y,\|\cdot\|_\Y)$ be a Banach space, and let $f:\mathcal{N}\to \Y$ be  $1$-Lipschitz. Then, there is an $L$-Lipscitz mapping $F:\ell_2^n\to \Y$ satisfying \begin{equation}\label{eq:sqrt in theorem}
\max_{a\in \NN} \|F(a)-f(a)\|_\Y\lesssim \frac{\sqrt{n}}{L}\e.\end{equation}
\end{theorem}

Here we prove that Theorem~\ref{thm:NS} is sharp up to the value of the implicit universal constant in~\eqref{eq:sqrt in theorem}.

\begin{theorem}\label{thm:L2 case} Fix $n\in \N$ and $L\ge 1$ with $6L\le\sqrt[4]{n}$. There exists $\e>0$ such that for any    $\e$-net $\mathcal{N}$ of $S^{n-1}$ there exists  a Banach space $(\Y,\|\cdot\|_\Y)$ and a $1$-Lipschitz mapping $f:\mathcal{N}\to \Y$ that satisfies
\begin{equation}\label{eq:max on N hilbert}
\max_{a\in \mathcal{N}} \|F(a)-f(a)\|_\Y\gtrsim \frac{\sqrt{n}}{L}\e,
\end{equation}
for {\em every} $L$-Lipschitz mapping $F:S^{n-1}\to \Y$.
\end{theorem}

The goal of~\cite{NS18-extension} is to improve Theorem~\ref{thm:bourgain} when $\X$ is the Schatten--von Neumann trace class $\mathsf{S}_p^n$ for some $n\in \N$ and $p\ge 1$; this is used in~\cite{NS18-extension} for a geometric application. The Euclidean setting of Theorem~\ref{thm:NS} is a byproduct of the investigations of~\cite{NS18-extension} that is valuable in its own right, but it is not needed for the purposes of~\cite{NS18-extension}. Another byproduct of~\cite{NS18-extension} is  an improvement of     Theorem~\ref{thm:bourgain} when $\X=\ell_p^n$ and  $p>1$; we will state this improvement later and also derive lower bounds when $\X=\ell_p^n$ and $p\in (1,\infty)\setminus \{2\}$  that are not as satisfactory as the optimal bound~\eqref{eq:max on N hilbert}. It would be interesting to obtain sharp results in this context as well. More substantially, it would be interesting (likely requiring a major new idea) to obtain sharp bounds when $\X=\mathsf{S}_p^n$ for $p\in [1,\infty]\setminus \{2\}$. For example, if $\X=\mathsf{S}_1^n$, then~\cite{NS18-extension}  shows that the conclusion of  Theorem~\ref{thm:bourgain} can be improved to
\begin{equation}\label{eq:3/4}
\max_{a\in \NN} \|F(a)-f(a)\|_\Y\lesssim \frac{n^{\frac{3}{2}}}{L}\e=\frac{\dim(\mathsf{S}_1^n)^{\frac34}}{L}\e.\end{equation}
We do not know how close~\eqref{eq:3/4} is to being optimal; our approach here does not seem to shed light on this.

\subsection*{Roadmap} Theorem~\ref{thm:l1} will be proved in Section~\ref{sec:l1 proof} and Theorem~\ref{thm:L2 case} will be proved in Section~\ref{thm:L2 case}. In Section~\ref{sec:ext intro} we will discuss in greater detail various aspects of the almost extension problem for Lipschitz functions, including comparing it to the classical Lipschitz extension problem, proving preliminary facts, and presenting intriguing open questions. Thus, Section~\ref{sec:ext intro} has an introductory component,  so those who prefer to  first read the proofs of Theorem~\ref{thm:l1} and Theorem~\ref{thm:L2 case} could skip directly to Section~\ref{sec:l1 proof} and Section~\ref{thm:L2 case}, respectively. The proof of Theorem~\ref{thm:L2 case} in Section~\ref{thm:L2 case} does use parts of Section~\ref{sec:ext intro}, namely Remark~\ref{rem:moment} and Proposition~\ref{prop:duality} therein, but these could be read independently. Also, Section~\ref{sec:lp variant} presents variants of Theorem~\ref{thm:l1} when the source space is $\ell_p^n$ for $1<p<2$, but the range $p>2$ is treated (less satisfactorily) in Section~\ref{sec:compactum}.

\section{Extension versus almost extension}\label{sec:ext intro} By proving Theorem~\ref{thm:bourgain}, Bourgain's achievement was not limited to merely establishing its statement. It contains the noteworthy conceptual realization that his formulation of the {\em almost extension problem} overcomes a  barrier that previously precluded certain applications. Theorem~\ref{thm:bourgain} shows  that by relaxing the traditional requirement that the restriction of $F$ to $\NN$ coincides with $f$ to the weaker requirement that the restriction of $F$ to $\NN$ is close to $f$, it is possible ensure that $F$ is $O(1)$-Lipschitz; this can fail badly if one insists that $F$ extends $f$, namely, in that case the Lipschitz constant must sometimes tend to $\infty$ with $n$. At the same time, some important applications necessitate  that $F$ is $O(1)$-Lipschitz, and they persist if $F$ is  $O(1)$-Lipschitz yet only sufficiently close to $f$ on $\NN$, as demonstrated decisively in~\cite{Bou87}.

The input of the classical {\em Lipschitz extension problem} is two metric spaces $(\MM,d_\MM),(\mathcal{Z},d_{\mathcal{Z}})$, as well as a subset  $\Omega\subset \MM$. The goal is to determine if there is $K\in [1,\infty)$ such that for every Lipschitz function $f:\Omega\to \mathcal{Z}$ there is a function $F:\MM\to \mathcal{Z}$ that extends $f$, i.e., $F(\omega)=f(\omega)$ for all $\omega\in \Omega$, and such that $\|F\|_{\Lip}\le K\|f\|_{\Lip}$. Here, and in what follows, $\|\cdot\|_{\Lip}$ denotes the Lipschitz constant. Let $\ee(\MM,\Omega;\mathcal{Z})$ denote the infimum over those $K$ for which the above extension phenomenon holds, with the convention that $\ee(\MM,\Omega;\mathcal{Z})=\infty$ if there is no such $K$. Write $\ee(\MM;\mathcal{Z})=\sup_{\Omega\subset \MM} \ee(\MM,\Omega;\mathcal{Z})$. Also, denote by $\ee(\MM,\Omega)$ the supremum of $\ee(\MM,\Omega;\Y)$ as $\Y$ ranges over all   Banach spaces, and define $\ee(\MM)=\sup_{\Omega\subset \MM} \ee(\MM,\Omega)$.

Finding geometric conditions on the triple $(\MM,\Omega,\mathcal{Z})$ which ensure that $\ee(\MM,\Omega;\mathcal{Z})<\infty$, and moreover obtaining  upper and lower bounds on $\ee(\MM,\Omega;\mathcal{Z})$, are important questions with applications in several areas. The techniques that were developed to extend Lipschitz functions exhibit a rich interplay between this area and other mathematical disciplines;  see the monograph~\cite{BB12} for an indication of (part of) the many results along these lines that have been obtained over the past century.

While  $\ee(\MM;\Omega)$ is defined in terms of extension of mappings that can take values in {\em all possible} Banach spaces, it has an intrinsic geometric characterization~\cite{Koz05,BB07-2,Oht09,AP20}. We will next explain a  generalization of this fact that applies to (but is more general than) the almost extension setting of Theorem~\ref{thm:bourgain}; this will  clarify the geometric task at hand, and will also be valuable for the proof of Theorem~\ref{thm:L2 case}. For the sake of brevity and in order to avoid the need to treat issues (measurability, working with the Lipschitz  free space over $\Omega$; see~\cite{Wea18}) that are by now well understood  but  could obscure the main geometric content, our presentation  (in Proposition~\ref{prop:duality}) will be confined to the case when $\Omega$ is  finite;  this  is all that we will use here (in our setting, $\Omega=\NN$ is a net in the unit sphere of a finite dimensional normed space, hence it is finite), but if the need to treat infinite subsets will arise in future investigations, then it will be mechanical (thanks to~\cite{Koz05,BB07-2,Oht09,AP20})  to adjust the characterization accordingly (using the Lipschitz free space instead of the Wasserstein-$1$ space).

Suppose that $(\MM,d_\MM)$ is a metric space and that $\Omega$ is a finite subset of $\MM$. Let $\W_1(\Omega)$ be the space of all the signed measures $\mu$ on $\Omega$ with vanishing total mass, namely $\mu(\Omega)=0$. Equivalently, $\W_1(\Omega)$ is the hyperplane in $\R^\Omega$ consisting of all those $\mu:\Omega\to \R$ that satisfy $\sum_{\omega\in \Omega}\mu(\omega)=0$. For each $\mu\in \W_1(\Omega)$ denote $\mu^+=\max\{\mu,0\}$ and $\mu^-=\max\{-\mu,0\}$, so that $\mu^+,\mu^-$ are disjointly supported nonnegative measures on $\Omega$ with $\mu^+(\Omega)=\mu^-(\Omega)$ and $\mu=\mu^+-\mu^-$. The set of all possible couplings of $\mu^+$ and $\mu^-$ is denoted $\Pi(\mu^+,\mu^-)$; it consists of all those nonnegative measures $\pi$ on $\Omega\times \Omega$ that satisfy $\sum_{\omega'\in \Omega}\pi(\omega,\omega')=\mu^+(\omega)$ and    $\sum_{\omega'\in \Omega}\pi(\omega',\omega)=\mu^-(\omega)$ for every $\omega\in \Omega$. The Wasserstein-1 norm of $\mu\in \W_1(\Omega)$ is defined by
\begin{equation}\label{eq:wasserstein norm}
\|\mu\|_{\W_1(\Omega)}\eqdef \inf_{\pi\in \Pi(\mu^+,\mu^-)} \iint_{\Omega\times \Omega} d_\MM(\omega,\omega')\ud\pi(\omega,\omega').
\end{equation}
This turns $\W_1(\Omega)$ into a Banach space. See~\cite{Vil03} for the  background and context on optimal transport, though, due to our finitary setting, essentially none of that material will be needed here; we only recall the Kantorovich--Rubinstein duality theorem (see~\cite[Theorem~1.14]{Vil03}) which states that
\begin{equation}\label{eq:KR}
\|\mu\|_{\W_1(\Omega)}=\sup_{\substack{\f:\Omega\to \R\\ \|\f\|_{\Lip}\le 1}} \int_\Omega \f\ud \mu.
\end{equation}
The geometric interpretation of the duality~\eqref{eq:KR} is that the unit ball of $\W_1(\Omega)$ is the following polyope
$$
B_{\W_1(\Omega)}=\mathrm{conv} \left\{\frac{\bd_\omega-\bd_{\omega'}}{d_\MM(\omega,\omega')}:\ \omega,\omega'\in \Omega\ \wedge\ \omega\neq\omega'\right\},
$$
where $\bd_\omega$ denotes the delta-mass at the point $\omega\in \Omega$. Equivalently, $\W_1(\Omega)$ is the dual of the Banach space consisting of all the functions $f:\Omega\to \R$ that satisfy $\sum_{\omega\in \Omega} f(\omega)=0$, equipped with the norm $\|f\|_{\Lip}$.

\begin{remark}\label{rem:moment} We record for ease of later use (in the proof of Theorem~\ref{thm:L2 case}) the following simple consequence of~\eqref{eq:KR}. Let $(\X,\|\cdot\|_\X)$ be a normed space and  $\Omega\subset \X$ a finite subset. The {\em moment} of a measure $\mu\in \W_1(\Omega)$ is
\begin{equation}\label{eq:moment def}
\text{\tt Moment}_\mu \eqdef \int_\Omega \omega\ud\mu(\omega)\in \X.
\end{equation}
The length of the vector $\text{\tt Moment}_\mu$ is at most the Wasserstein-1 norm of $\mu$. Indeed, if $x^*\in S_{\X^*}$ is a normalizing functional of $\text{\tt Moment}_\mu$, i.e., $x^*(\text{\tt Moment}_\mu)=\|\text{\tt Moment}_\mu\|_\X$, then since $\|x^*\|_{\Lip}=\|x^*\|_{\X^*}=1$,
\begin{equation}\label{eq:moment wasserstein}
\left\|\text{\tt Moment}_\mu\right\|_\X= \int_\Omega x^*\ud\mu\le \|\mu\|_{\W_1(\Omega)}.
\end{equation}
\end{remark}

\begin{proposition}\label{prop:duality} Suppose that  $(\MM,d_\MM)$ is a metric space and fix a finite subset  $\Omega\subset \MM$. Fix also $L\ge 1$ and a function $\alpha:\Omega\to [0,\infty)$. Then, the following statements are equivalent.
\begin{enumerate}
\item For every Banach space $(\Y,\|\cdot\|_\Y)$ and every $1$-Lipschitz mapping $f:\Omega\to \Y$ there is an $L$-Lipschitz mapping $F:\MM\to \Y$ that satisfies $\|F(\omega)-f(\omega)\|_\Y\le \alpha(\omega)$ for every $\omega\in \Omega$.
    \item For every probability measure $\nu$ supported on $\Omega$ and every $x\in \MM$ there is $\mu_x=\mu_x(\nu)\in \W_1(\Omega)$ such that $\|\mu_x-\mu_y\|_{\W_1(\Omega)}\le Ld_\MM(x,y)$ for every $x,y\in \MM$, and $\|\mu_\omega-(\bd_\omega-\nu)\|_{\W_1(\Omega)}\le \alpha(\omega)$ for every $\omega\in \Omega$.
\end{enumerate}
\end{proposition}
Observe that if part~${(\it 2)}$ of Proposition~\ref{prop:duality}  holds for some probability measure $\nu$ on $\Omega$, then it also holds for any other probability measure $\nu'$ on $\Omega$, as seen by  considering $\{\mu_x(\nu')=\mu_x(\nu)+\nu-\nu'\}_{x\in \MM}$. If $\alpha\equiv 0$, then Proposition~\ref{prop:duality} is  the characterization of classical Lipschitz extension in~\cite{Koz05,BB07-2,Oht09,AP20}, while the almost extension setting of Theorem~\ref{thm:bourgain} and Theorem~\ref{thm:L2 case} corresponds to the special case of Proposition~\ref{prop:duality} when $\alpha$ is the constant function which equals a positive  multiple of $\e$, and $\Omega$ is the $\e$-net $\NN$.

\begin{remark} To elucidate the geometric meaning of part~${(\it 2)}$ of Proposition~\ref{prop:duality}, consider its variant in which $\{\mu_x\}_{x\in \MM}$ are probability measures on $\Omega$ instead of elements of $\W_1(\Omega)$, and the analogous requirement is
\begin{equation}\label{eq:convex hall variant}
\forall x,y\in \MM, \qquad \|\mu_x-\mu_y\|_{\W_1(\Omega)}\le Ld_\MM(x,y)\qquad\mathrm{and} \qquad \forall\omega\in \Omega,\qquad \|\mu_\omega-\bd_\omega\|_{\W_1(\Omega)}\le \alpha(\omega).
\end{equation}
 The ensuing proof  of Proposition~\ref{prop:duality} shows mutatis mutandis that this is equivalent to the same assertion as part~${\it (1)}$ of Proposition~\ref{prop:duality}, with the additional requirement that $F$ takes values in the {\em convex hull} of $f(\Omega)$.  One can interpret the family of probability measures   $\{\mu_x\}_{x\in \MM}$ as  a ``randomized rounding scheme'' that associates to each point in the ambient space $\MM$ a random element of the subset $\Omega$. The first condition in~\eqref{eq:convex hall variant} is a consistency requirement for this rounding scheme, ensuring that nearby points distribute themselves over $\Omega$ in manners that are close to each other in the sense of transportation cost. The second condition in~\eqref{eq:convex hall variant} is that each point of $\Omega$ distributes itself over $\Omega$ in a manner that is close (in the sense of transportation cost) to the delta mass at that point; the case $\alpha\equiv 0$, which corresponds to Lipschitz extension (with the target constrained to be the convex hall of $\Omega$), is the more stringent requirement that each element of $\omega$ is ``rounded'' to itself with probability $1$. In actuality, part~${(\it 2)}$ of Proposition~\ref{prop:duality}  associates to each point in the ambient space $\MM$ an element of $\W_1(\Omega)$, i.e., the difference of two disjointly supported  positive measures on $\Omega$ of equal total mass; this makes more use of the vector space structure of $\Y$, rather than only considering its convex structure, but we still think of the measures $\{\mu_x\}_{x\in \MM}$ in part~${(\it 2)}$ of Proposition~\ref{prop:duality}  as a form of ``stochastic retraction'' of $\MM$ onto $\Omega$. We warn, however, that while in full generality the behaviors of the  extension problem into $\Y$ and the extension problem into $\mathrm{conv}(f(\Omega))\subset \Y$ are demonstrably different~\cite{Nao20}, at present the relation between the two questions is for the most part shrouded in mystery; these issues are discussed in greater detail in the forthcoming work~\cite{Nao20}.

\end{remark}

\begin{proof}[Proof of Proposition~\ref{prop:duality}] The implication ${\it (1)}\implies {\it (2)}$ is a direct application of ${\it (1)}$ to the case $\Y=\W_1(\Omega)$ and $f(\omega)=\bd_\omega-\nu$ for $\omega\in \Omega$, where we set $F(x)=\mu_x$ for $x\in \MM$.

To justify   the reverse implication ${\it (2)}\implies {\it (1)}$, define for every $x\in \MM$, $$F(x)\eqdef\int_\Omega f\ud\mu_x+\int_\Omega f\ud\nu\in \Y.$$
 Fix $x,y\in \MM$ and take $\pi\in \Pi((\mu_x-\mu_y)^+,(\mu_x-\mu_y)^-)$ for which
 \begin{equation}\label{eq:choose pi}
 \iint_{\Omega\times \Omega} d_\MM(\omega,\omega')\ud\pi(\omega,\omega')=\|\mu_x-\mu_y\|_{\W_1(\Omega)}\le Ld_\MM(x,y).\end{equation}  Then, $F$ is $L$-Lipschitz because
\begin{align*}
\|F(x)-F(y)\|_{\Y}&=\Big\|\int_{\Omega} f\ud(\mu_x-\mu_y)^+-\int_\Omega f\ud(\mu_x-\mu_y)^-\Big\|_\Y\\&=\Big\|\iint_{\Omega\times \Omega} \big(f(\omega)-f(\omega')\big)\ud\pi(\omega,\omega')\Big\|_\Y\\ &\le \iint_{\Omega\times \Omega} \|f(\omega)-f(\omega')\|_\Y\ud\pi(\omega,\omega')\\&\le \iint_{\Omega\times \Omega} d_\MM(\omega,\omega')\ud\pi(\omega,\omega')\stackrel{\eqref{eq:choose pi}}{\le} Ld_\MM(x,y),
\end{align*}
where the second equality holds because $\pi$ is a coupling of $(\mu_x-\mu_y)^+$ and $(\mu_x-\mu_y)^-$.
Next, fix $\omega\in \Omega$ and take $\sigma\in \Pi((\mu_\omega-(\bd_\omega-\nu))^+,(\mu_\omega-(\bd_\omega-\nu))^-)$  for which \begin{equation}\label{eq:choose sigma}
\iint_{\Omega\times \Omega} d_\MM(\omega',\omega'')\ud\sigma(\omega',\omega'')=\|\mu_\omega-(\bd_\omega-\nu)\|_{\W_1(\Omega)}\le \alpha(\omega).
\end{equation}
Then,
\begin{align*}
\|F(\omega)-f(\omega)\|_\Y&=\Big\|\Big(\int_\Omega f\ud\mu_x+\int_\Omega f\ud\nu\Big)-\int_\Omega f\ud\bd_\omega\Big\|_\Y\\&=\Big\|\int_\Omega f\ud\big(\mu_\omega-(\bd_\omega-\nu)\big)^+-\int_\Omega f\ud \big(\mu_\omega-(\bd_\omega-\nu)\big)^- \Big\|_\Y\\
&= \Big\|\iint_{\Omega\times \Omega} \big(f(\omega')-f(\omega''t)\big)\ud\sigma(\omega',\omega'')\Big\|_\Y\\&\le \iint_{\Omega\times \Omega} \|f(\omega')-f(\omega'')\|_\Y\ud\sigma(\omega',\omega'')\\&\le  \iint_{\Omega\times \Omega} d_\MM(\omega',\omega'')\ud\sigma(\omega',\omega'')\stackrel{\eqref{eq:choose sigma}}{\le} \alpha(\omega).\tag*{\qedhere}
\end{align*}
\end{proof}

Both parts of Proposition~\ref{prop:duality} are beneficial viewpoints for either proving the existence of an (almost) extension, or for proving impossibility results; the latter  is the topic of the present work. Our proof of Theorem~\ref{thm:bourgain} directly rules out the existence of an almost extension as in part~${(\it 1)}$ of Proposition~\ref{prop:duality}, while our proof of Theorem~\ref{thm:L2 case} rules out the existence of a stochastic retraction as in part~${(\it 2)}$ of Proposition~\ref{prop:duality}.

There is a universal constant $c>0$ such that every finite dimensional normed space $\X$ satisfies  \begin{equation}\label{eq:quote JLS}\dim(\X)^c\lesssim \ee(\X)\lesssim \dim(\X).\end{equation}
The upper bound on $\ee(\X)$ in~\eqref{eq:quote JLS} is a classical theorem of~\cite{JLS86} and the lower bound on $\ee(\X)$ in~\eqref{eq:quote JLS} is due to the forthcoming work~\cite{Nao20}. It is a  major problem  to evaluate the asymptotic behavior (as $n\to \infty$) of $\ee(\X)$ for specific $n$-dimensional normed spaces $\X$, but this goal resisted efforts for a long time. The only normed space for which this was accomplished (up to lower order factors) is $\ell_\infty^n$, namely
\begin{equation}\label{eq:l infty}
\sqrt{n}\lesssim \ee(\ell_\infty^n)\lesssim \sqrt{n\log n}.
\end{equation}
The lower bound on $\ee(\ell_\infty^n)$ in~\eqref{eq:l infty} is a classical combination of~\cite{Sob41} and~\cite{Lin64}, and the upper bound on $\ee(\ell_\infty^n)$  in~\eqref{eq:l infty} is due to~\cite{Nao17-SODA}. Remarkably, even the Euclidean case $\X=\ell_2^n$ remains a tantalizing mystery, with the currently best known bounds being
\begin{equation}\label{eq:ext euclidean known}
\sqrt[4]{n}\lesssim \ee(\ell_2^n)\lesssim \sqrt{n}.
\end{equation}
The upper bound on $\ee(\ell_2^n)$ in~\eqref{eq:ext euclidean known} is due to~\cite{LN05} and the lower bound on $\ee(\ell_2^n)$ in~\eqref{eq:ext euclidean known} is due to~\cite{MN13} (building on ideas of~\cite{Kal04,Kal12}). Determining the asymptotic behavior of the maximum of $\ee(\X)$ over all possible $n$-dimensional normed spaces $\X$ is also a major longstanding open question; at present, the best known bounds are that it is at most a constant multiple of $n$ and at least a constant multiple of $\sqrt{n}$.

The above discussion demonstrates that even in the seemingly ``nice'' setting of finite dimensional normed (even Euclidean) spaces\footnote{Lipschitz non-extension phenomena were derived  in many other settings as well. For example, see~\cite{NR17} for the currently best-known general bounds in terms of the cardinality of $\Omega$ as well as~\cite{MM10} for algorithmic ramifications of such bounds. Lipschitz non-extension theorems in the context of geometric group theory can be found in~\cite{BF09,RW10,NP11,BLP16}.}, one cannot hope to be able to extend any $1$-Lipschitz mapping to a mapping whose Lipschitz constant does not tend to $\infty$ with the dimension (moreover, the rate must sometime be of power-type). In contrast, in Theorem~\ref{thm:bourgain} Bourgain specifies a firm upper bound $L$ that the Lipschitz constant of $F$ must not exceed, and shows that this is achievable by no longer insisting that $F$ extends $f$. Having such control on the Lipschitz constant of $F$ is important for certain applications, most notably the alternative tradeoff of Theorem~\ref{thm:bourgain} is used crucially in the proof of {\em Bourgain's discretization theorem}~\cite{Bou87}, which is a quantitative version of an important rigidity theorem of Ribe~\cite{Rib76}.

When one is given a  discretely-defined $1$-Lipschitz mapping $f$, it is often desirable to obtain a mapping $F$ that ``mimics'' $f$ yet is defined on the entire ambient normed space, so as to be able to apply continuous methods, e.g.~using differentiation or harmonic analysis. Any such reasoning that needs  dimension-independent bounds on the derivatives of $F$ requires good control on $\|F\|_{\Lip}$. The proof of Ribe's theorem in~\cite{Bou87} would utterly fail if the norm of the derivative of $F$ were unbounded as $n\to \infty$, because Ribe's theorem is all about dimension-independent bounds (its conclusion concerns the finite representability of an infinite dimensional Banach space). At the same time, by~\cite{Bou87} the desired statement (Ribe's theorem, and even a quantitative version thereof) follows if $F$ is sufficiently close to $f$ on the net $\NN$ (per the conclusion of Theorem~\ref{thm:bourgain}), provided that  $f$ is $O(1)$-bi-Lipschitz, and $\NN$ is fine enough, i.e., $\e$ is sufficiently small (the requirement of~\cite{Bou87} is $\e\le \exp(-\exp(O(n)))$, which remains the best known bound  in full generality; finding the optimal asymptotic behavior here is a fascinating open question; see~\cite{GNS12,Ost13,LN13,HLN16,HN19}). Other applications of Bourgain's almost extension theorem appear in~\cite{Bou87,GNS12,NR17,NS18-extension}, some of which build on its simplified proof in~\cite{Beg99} (but we checked that they could be deduced also using the original proof in~\cite{Bou87}); not all of these applications belong to the above framework, e.g.~\cite[Corollary~22]{NR17} is a statement in approximation theory.

Beyond his demonstration of the  highly significant application to Ribe's rigidity theorem, Bourgain took note of the potential of the aforementioned ``twist''  in Theorem~\ref{thm:bourgain} on the classical Lipschitz extension problem, wherein one can specify the desired bound on $\|F\|_{\Lip}$, by stating on page~158 of~\cite{Bou87} that the existence of such a mapping ``may already have considerable implications.'' He proceeded to further illustrate this by deducing Theorem~\ref{thm:BourDvo} below on almost ellipsoidal sections. It is beneficial to recall this specific application here because of its direct relevance to the present work, namely this is how a power-type impossibility result in the context of Theorem~\ref{thm:bourgain} is obtained in~\cite{Bou87}.

Prior to stating Theorem~\ref{thm:BourDvo}, we introduce the following (somewhat ad-hoc, but convenient) terminology that facilitates  discussion of the parameters that occur in Theorem~\ref{thm:bourgain}.

\begin{definition}[almost extension pair]\label{def:almost extension pair} Let $(\X,\|\cdot\|_\X)$ be a normed space. A pair $(L,A)\in [1,\infty)\times [1,\infty)$ is an $\X$-{\em almost extension pair} if  for every $\e>0$, every normed space $(\Y,\|\cdot\|_\Y)$, every $\e$-net $\NN$ of the unit sphere of $\X$, and every $1$-Lipschitz mapping $f:\NN\to \Y$, there exists an $L$-Lipschitz mapping $F:\X\to \Y$ that satisfies $$\sup_{a\in \NN}\|F(a)-f(a)\|_\Y\le A\e.$$
\end{definition}

Let $(\MM,d_\MM),(\MM',d_{\MM'})$ be  metric spaces. Suppose  that they are bi-Lipschitz equivalent to each other, namely that there exist a bijection $\psi:\MM\to \MM'$,  a distortion $D\ge 1$, and a scaling factor $s>0$ such that $sd_\MM(x,y)\le d_\MM'(x,y)\le Ds d_\MM(x,y)$ for all $x,y\in \MM$. The infimum over those $D\ge 1$ for which this holds is called the bi-Lipschitz distance between $(\MM,d_\MM)$ and $(\MM',d_{\MM'})$, and it is denoted (when the metrics are clear from the context) by $d_\Lip(\MM,\MM')$. Set $d_\Lip(\MM,\MM')=\infty$ if $(\MM,d_\MM),(\MM',d_{\MM'})$ are not bi-Lipschitz equivalent. The bi-Lipschitz distortion of $\MM$ in $\MM'$ is defined by $\cc_{\MM'}(\MM)=\inf_{\Omega\subset \MM'} d_\Lip(\MM,\Omega)$.

If  $(\X,\|\cdot\|_\X)$ and $(\Y,\|\cdot\|_\Y)$ are  two isomorphic normed spaces, then their Banach--Mazur distance, denoted $d_\BM(\X,\Y)$, is the infimum over those $D\ge 1$ for which there exists a linear operator $T:\X\to \Y$ such that $\|x\|_\X\le \|Tx\|_\Y\le D\|x\|_\X$ for all $x\in X$. If $\dim(\X)=\dim(\Y)=n<\infty$, then by compactness  $d_\BM(\X,\Y)$ is attained, and by a standard application of the Rademacher differentiation theorem (see e.g.~\cite{BL00}) we have $d_\BM(\X,\Y)=d_\Lip(\X,\Y)$. If $1\le \alpha<\infty$ and $d_\BM(\ell_2^{n},\X)\le \alpha$, then $\X$ is said to be  $\alpha$-Hilbertian; geometrically, this is equivalent to the assertion that there exists an ellipsoid $\boldsymbol{\mathcal{E}}\subset \X$ such that $\boldsymbol{\mathcal{E}}\subset B_\X\subset \alpha \boldsymbol{\mathcal{E}}$.

\begin{theorem}[\cite{Bou87}]\label{thm:BourDvo} For any $D,K\ge 1$ there exists  $\e\asymp1/\max\{D^2,K\}$ with the following property. Suppose that $n\in \N$ and $(D,K/D)$ is an $\ell_2^n$-almost extension pair. Let $\NN$  be an $\e$-net of $S^{n-1}$ and let $\Y$ be a normed space for which $\cc_\Y(\NN)\le D$. Then, $\Y$ has  a $2$-Hilbertian subspace $\mathbf{H}\subset \Y$ such that
$
\dim(\mathbf{H})\gtrsim n/D^4.
$
\end{theorem}

Theorem~\ref{thm:BourDvo} is stated as~\cite[Theorem~2]{Bou87} only in the special case $K\asymp n$, because this is the statement of Theorem~\ref{thm:bourgain}, which was the only almost extension theorem that was available at the time; by Theorem~\ref{thm:NS}, we now know  that we can actually take $K\asymp\sqrt{n}$ here.\footnote{The especially significant case $D\lesssim 1$ of Theorem~\ref{thm:BourDvo}, together  with Theorem~\ref{thm:NS}, has the following  geometric interpretation. For some $\e\asymp 1/\sqrt{n}$, if an $\e$-net of $S^{n-1}$ embeds with distortion $O(1)$ into a normed space $\Y$, then the unit ball of $\Y$ has an almost ellipsoidal section of dimension proportional to $n$; this was obtained in~\cite{Bou87} with the weaker bound $\e\asymp 1/n$. It remains open to determine the slowest possible  rate at which $\e$ may tend to $0$ with $n$ such that the above conclusion still holds. If one aims to get an almost ellipsoidal section of full dimension $n$, rather than allowing its dimension to be a smaller proportion of $n$, then this is still possible, but now the best known  bound~\cite{Bou87} is much weaker, namely this holds for some $\e\ge \exp(-\exp(O(n\log n)))$. See~\cite{Bou87,GNS12,Ost13,LN13,HLN16,HN19,OR20} for more on these intriguing and longstanding discretization issues.} However, the above formulation is exactly what the proof in~\cite{Bou87} gives. The choice to discuss $2$-Hilbertian subspaces was made arbitrarily in~\cite{Bou87}, though for any $\alpha>1$ an $\alpha$-Hilbertian version follows formally; see~\cite{FLM77}.

Apart from the intrinsic geometric value of Theorem~\ref{thm:BourDvo} (see e.g.~\cite{LM93,GM01} for  surveys of the central theme in modern Banach space theory of finding almost ellipsoidal sections), it was used in~\cite{Bou87} (see equation~(5) there) as follows  to deduce an impossibility result for almost extension. Fix $n\in \N$ and $p>2$. By considering (a suitable rescaling of) the classical Mazur map~\cite{Maz29,BL00} we see that there is a  mapping $\phi:S^{n-1}\to S_{\ell_p^n}$ that satisfies the following bounds for any $0<\e<1$ and any $\e$-net $\NN$ of $S^{n-1}$.

\begin{equation*}
\forall a,b\in \NN,\qquad \|a-b\|_{\ell_2^n}\le \|\phi(a)-\phi(b)\|_{\ell_p^n}\lesssim p\|a-b\|_{\ell_2^n}^{\frac{2}{p}}\le p\e^{\frac{2}{p}-1}\|a-b\|_{\ell_2^n},
\end{equation*}
Hence, if we choose $p= 2+1/\log(2/\e)$, then it follows that $\cc_{\ell_p^n}(\NN)=O(1)$.  By taking $D,K\ge 1$ such that $(D,K/D)$ is an $\ell_2^n$-almost extension pair and $D$ is a large enough universal constant, Theorem~\ref{thm:BourDvo} implies that if $\e\asymp 1/\max\{D^2,K\}\asymp 1/K$, then  $\ell_p^n$ has a $2$-Hilbertian subspace $\mathbf{H}$ with $\dim(\mathbf{H})\gtrsim n$. However, the maximal dimension of almost ellipsoidal sections of the unit ball of $\ell_p^n$ was studied in~\cite{BDGJN77}, where it was shown that the above conclusion entails that \begin{equation}\label{eq:2/p}
\dim(\mathbf{H})\lesssim n^{\frac{2}{p}}.
\end{equation} So, $(\log n)/\log(2/\e)\asymp(1-2/p)\log n\lesssim 1$, and therefore we have $K\gtrsim n^c$ for some universal constant $c>0$.

Recalling that the premise of the above discussion was that $(D,K/D)$ is an $\ell_2^n$-almost extension pair, this reasoning of~\cite{Bou87} shows that the factor $n$ in the conclusion~\eqref{eq:bourgain conclusion} of Theorem~\ref{thm:bourgain} cannot be improved to $o(n^c)$. By optimizing the above reasoning, one sees that it is possible to  take $c=\frac14-o(1)$ as $n\to \infty$. However, we checked that if instead of using Theorem~\ref{thm:BourDvo} as a ``black box'' one incorporates its proof and the proof of~\eqref{eq:2/p} in~\cite{BDGJN77} into this reasoning in an optimal way, then it is possible to obtain  the better value $c=\frac12-o(1)$ as $n\to \infty$. We omit the details of this optimization because  it does not yield the nearly optimal bound of Theorem~\ref{thm:l1}. It is important to note that this drawback is inherent to Bourgain's approach, since the source space is $\ell_2^n$, so by Theorem~\ref{thm:NS} it cannot yield a value of $c$ that is larger than $\frac12$.

\begin{remark} The above approach of~\cite{Bou87}  to proving impossibility of almost extension (hence, in particular, also impossibility of actual extension) is an interesting idea that, to the best of our knowledge, differs from all other non-extension results that appeared in the literature. As such, it warrants further scrutiny as a potential route towards related open questions; we defer such investigations to future work. The aforementioned barrier $c\le \frac12$ that makes this approach unsuitable as a potential route towards Theorem~\ref{thm:l1} stems from its strong dependence on Euclidean symmetries, both in terms of the proof in~\cite{Bou87}, as well as its appeal to~\cite{BDGJN77,FLM77}. Even in the Euclidean setting, we do not see how this approach could possibly yield the sharp result of Theorem~\ref{thm:L2 case}, namely it seems that the loss of an unbounded lower order factor is inherent to the route that is taken in~\cite{Bou87}.
%Another drawback of the approach~\cite{Bou87} is its reliance on the low dimensionality   of the target space. Our proof of Theorem~\ref{thm:l1} %uses a bi-Lipschitz invariant that is insensitive to the dimension of the target; see Remark~\ref{}.
\end{remark}

\subsection{Bi-Lipschitz invariance and the sphere compactum}\label{sec:compactum} It is natural to try to prove an almost extension theorem for a Banach space $\X$, i.e., to determine specific $\X$-almost extension pairs $(L,A)$,  by proving such a theorem for some Banach space $\X'$ with a good bound on $d_{\BM}(\X,\X')$, and somehow transferring that theorem back to $\X$ itself. In the same vein, it is natural to try to deduce an impossibility result for almost extension from nets in $S_\X$, from such a result for $\X'$. The following lemma is a small step in that direction that yields in some cases the best bounds that we can currently achieve. As we will see later, there is potential for great improvement here that relates to natural open questions of independent interest.

\begin{lemma}\label{lem:pass between spaces} Fix $L,A,D,K\ge 1$. Suppose that $(\X,\|\cdot\|_\X), (\Y,\|\cdot\|_\Y)$ are normed spaces and $(L,A)$ is a $\Y$-almost extension pair. Let $\mathbf{V}\subset \Y$ be a subspace of $\Y$ for which $d_{\BM}(\X,\mathbf{V})< D$ and there is a $K$-Lipschitz mapping $\psi:S_\Y\to \mathbf{V}$ such that $\psi(y)=y$ for any $y\in S_{\mathbf{V}}=\mathbf{V}\cap S_{\Y}$. Then, $(90KD^2L,9KA)$ is an $\X$-almost extension pair.
\end{lemma}

Prior to proving Lemma~\ref{lem:pass between spaces}, we will derive some of its consequences. For $1\le p\le 2\le q<\infty$, the (Gaussian) type $p$ and cotype $q$ constants~\cite{MP76} of a normed space $(\X,\|\cdot\|_\X)$, denoted $T_p(\X)$ and $C_q(\X)$, respectively,  are the infimal $T,C\ge 1$ such that for every $n\in \N$ and every $x_1,\ldots,x_n\in \X$ the following inequalities hold, where $\mathsf{g}_1,\ldots\mathsf{g}_n$ are independent standard Gaussian random variables.
\begin{equation}\label{eq:type cotype def}
\frac{1}{C}\bigg(\sum_{j=1}^n \|x_j\|_\X^q\bigg)^{\frac{1}{q}}\le \left(\E\bigg[\Big\|\sum_{j=1}^n \mathsf{g}_j x_j\Big\|_\X^2\bigg]\right)^{\frac12}\le T\bigg(\sum_{j=1}^n \|x_j\|_\X^p\bigg)^{\frac{1}{p}}.
\end{equation}

\begin{corollary}\label{cor:use FLM} There exists a universal constant $\kappa>0$ with the following property. Fix $L,A\ge 1$ and $q\ge 2$. Let $(\X,\|\cdot\|_\X)$ be a finite dimensional normed space of sufficiently large dimension in the following sense
\begin{equation}\label{eq:X has large dimension}
\dim(\X)\ge \big(\kappa T_2(\X)^2C_q(\X)L^2\big)^q.
\end{equation}
Suppose  that $(L,A)$ is an $\X$-almost extension pair. Then,
\begin{equation}\label{eq:desired cotype}
A\gtrsim \frac{\dim(\X)^{\frac{1}{q}}}{T_2(\X)^2C_q(\X)L}.
\end{equation}
In the special $\X=\ell_q^n$ for some $n\in \N$, the above assumptions imply that if $n\ge (\kappa \sqrt{q}L^2)^q$, then
\begin{equation}\label{eq:A lower lq}
A\gtrsim \frac{n^{\frac{1}{q}}}{\sqrt{q}L}.
\end{equation}
\end{corollary}

\begin{comment}
\begin{theorem} Let $(\X,\|\cdot\|_\X)$ be a finite dimensional normed space. Suppose that $L,A\ge 1$ are such that $(L,A)$ is an $\X$-almost extension pair. Then, there is a universal constant $\kappa>0$ such that for every $q\ge 2$ we have
\begin{equation}\label{eq:desired cotype}
A\gtrsim \frac{\dim(\X)^{\frac{1}{q}}}{T_2(\X)^2C_q(\X)L}-\kappa L.
\end{equation}
\end{theorem}
\end{comment}

\begin{proof} By~\cite{FLM77}, there exists a subspace $\mathbf{V}\subset \X$ that satisfies
\begin{equation}\label{eq:dim V lower}
\dim(\mathbf{V})\gtrsim \frac{\dim(\X)^{\frac{2}{q}}}{C_q(\X)^2}\qquad\mathrm{and}\qquad d_{\BM}(\ell_2^{\dim(\mathbf{V})},\mathbf{V})< 2.
\end{equation}
In particular, this implies that $C_2(\mathbf{V})\le  C_2(\ell_2^{\dim(\mathbf{V})})\le 2$. By~\cite{Mau74} the formal identity  $\mathsf{Id}_{\mathbf{V}}:\mathbf{V}\to \mathbf{V}$ extends to a linear mapping $\psi:\X\to \mathbf{V}$ whose operator norm satisfies $\|\psi\|_{\X\to \X}=\|\psi\|_{\Lip}\le T_2(\X)C_2(\mathbf{V})\le 2T_2(\X)$.  Lemma~\ref{lem:pass between spaces} (with $K=2T_2(\X)$ and $D=2$) now shows that $(L',A')$ is an $\ell_2^{\dim(\mathbf{V})}$-almost extension pair, where
\begin{equation}\label{eq:L'A'}
L'\eqdef 720T_2(\X)L\qquad\mathrm{and} \qquad A'\eqdef  18T_2(\X)A.
\end{equation}
 Due to Theorem~\ref{thm:L2 case}, if
 \begin{equation}\label{eq:condition from L2 theorem}
 4320T_2(\X)L\stackrel{\eqref{eq:L'A'}}{=}6L'\le \sqrt[4]{\dim(\mathbf{V})},
 \end{equation}
then necessarily
$$
18T_2(\X)A\stackrel{\eqref{eq:L'A'}}{=}A'\stackrel{\eqref{eq:max on N hilbert}}{\gtrsim} \frac{\sqrt{\dim(\mathbf{V})}}{L'}\stackrel{\eqref{eq:L'A'}}{=}\frac{\sqrt{\dim(\mathbf{V})}}{720T_2(\X)L}\stackrel{\eqref{eq:dim V lower}}{\gtrsim} \frac{\dim(\X)^{\frac{1}{q}}}{T_2(\X)C_q(\X)L}.
$$
This establishes the desired lower bound~\eqref{eq:desired cotype}, because due to the assumption~\eqref{eq:X has large dimension} and  the first inequality in~\eqref{eq:dim V lower}, the requirement~\eqref{eq:condition from L2 theorem} holds provided $\kappa$ is a large enough universal constant.

If $n\in \N$ satisfies $n\ge (\kappa \sqrt{q}L^2)^q$, then the bound~\eqref{eq:A lower lq} for the special case $\X=\ell_q^n$ follows by substituting into the above reasoning the fact that in this case the first inequality in~\eqref{eq:dim V lower} can be improved to $$\dim(\mathbf{V})\gtrsim qn^{\frac{2}{q}},$$ which is sharp up to the implicit universal constant by~\cite{BDGJN77}; this lower bound on $\dim(\mathbf{V})$ is stated in~\cite[Remark~5.7]{MS86}, and it follows by substituting the estimate in~\cite[Lemma~2(4)]{SZ90} into~\cite{Mil71}.
\end{proof}

The following two interesting open questions arise naturally from Corollary~\ref{cor:use FLM}.

\begin{question}\label{Q:n^c} For $L>1$, is there a sequence $\{a_n(L)\}_{n=1}^\infty\subset \R$ satisfying $\lim_{n\to \infty} a_n(L)=\infty$ with the property that if $A\ge 1$ and $(\X,\|\cdot\|_\X)$ is a finite dimensional normed space  for which $(L,A)$ is an $\X$-almost extension pair, then necessarily $A\ge a_{\dim(\X)}(L)$? If so, then in analogy to the first inequality in~\eqref{eq:quote JLS}, could one even take $a_n(L)=\eta(L) n^c$ for some universal constant $c>0$ and some $\eta(L)>0$ that depends only on $L$?
\end{question}

\begin{question}\label{Q:q>2} By~\cite{NS18-extension}, if $q\ge 2$ and $L>C$ for a sufficiently large universal constant $C>0$, then  $(L,K_q \sqrt{n})$ is an $\ell_q^n$-almost extension pair, where $K_q>0$ depends only on $q$. The best lower bound that we can currently derive in this context is~\eqref{eq:A lower lq}. What is the correct asymptotic behavior here?
\end{question}

For $1< p< 2$ we have $T_2(\ell_p^n)\le n^{\frac{1}{p}-\frac12}$ and $C_2(\ell_p^n)\lesssim 1$ for all $n\in \N$ (see~\cite{MP76}). So,  if $\sqrt{n}\ge  n^{\frac{2}{p}-1}(\alpha L)^2$ for some universal constant $\alpha>0$, then by Corollary~\ref{cor:use FLM}   if $(L,A)$ is an $\ell_p^n$-almost extension pair, then
\begin{equation}\label{eq:3/2 bound}
A\gtrsim \frac{n^{\frac32-\frac{2}{p}}}{L}.
\end{equation}
The lower bound~\eqref{eq:3/2 bound} is not vacuous only when $\frac43<p< 2$, and for this range  the above restriction on $n$ is indeed satisfied when $n$ is large enough, namely, the requirement for~\eqref{eq:3/2 bound} becomes $n\ge (\alpha L)^{p/(3p-4)}$. However, in Section~\ref{sec:lp variant} we will prove that if $n\ge n_0(L,p)$ and $(L,A)$ is an $\ell_p^n$-almost extension pair, then
\begin{equation}\label{eq:lp variant}
A\gtrsim \frac{Ln^{\frac{1}{p}}}{e^{\beta\sqrt{(\log n)\log (4L)}}},
\end{equation}
where $\beta>0$ is a universal constant. So,  Lemma~\ref{lem:pass between spaces} is quite weak when $p< 2$, because  when  $L=O(1)$ the bound~\eqref{eq:3/2 bound} is better than~\eqref{eq:lp variant} (which we will prove via an entirely different route)   only if $2-p\lesssim 1/\sqrt{\log n}$.

To formalize the type of transference phenomenon that Lemma~\ref{lem:pass between spaces} aims to achieve (with, as we have seen above, partial success), it is convenient to introduce the following ad-hoc notation.  For a normed space $(\X,\|\cdot\|_\X)$ and $L\ge 1$, let $A_L(\X)$ denote the infimum over those $A\ge 1$ such that $(L,A)$ is an $\X$-almost extension pair. Thus, using this notation, Theorem~\ref{thm:bourgain} and Theorem~\ref{thm:NS} state, respectively, that for $L>C$ we have $A_L(\X)\lesssim \dim(\X)/L$ and, if $\X$ is a Hilbert space, then $A_L(\X)\lesssim \sqrt{\dim(\X)}/L$. Also, Lemma~\ref{lem:pass between spaces} states that
\begin{equation}\label{eq:reqwrite lemma with modulus}
A_{90KD^2L}(\X)\le 9K A_L(\Y),
\end{equation}
where we are using the notation  in the statement of Lemma~\ref{lem:pass between spaces}.

The few estimates that are currently known for $A_L(\X)$ are all consistent with the possibility that there exists an unbounded increasing modulus $\f:[1,\infty)\to [1,\infty)$ (possibly even $\f(L)\lesssim L$ for all $L\ge 1$) such that for every  two finite dimensional normed spaces $(\X,\|\cdot\|_\X),(\Y,\|\cdot\|_\Y)$ with $\dim(\X)=\dim(\Y)$ we have
\begin{equation}\label{eq:is approximation BM invariant}
A_L(\X)\lesssim d_{\mathrm{BM}}(\X,\Y) A_{\f(L)}(\Y).
\end{equation}

We do not have sufficient evidence to conjecture that~\eqref{eq:is approximation BM invariant} holds in full generality, but if that were so, then it would be valuable to prove it. For example, by John's theorem~\cite{Joh48}, if~\eqref{eq:is approximation BM invariant} were true, then Theorem~\ref{thm:bourgain} would follow from Theorem~\ref{thm:NS}. Also, if~\eqref{eq:is approximation BM invariant} were true, then~\eqref{eq:3/4} would follow from Theorem~\ref{thm:NS}  because the Banach--Mazur distance between $\mathsf{S}_1^n$ and the Hilbert space $\mathsf{S}_2^n$ is at most $\sqrt{n}$. In terms of lower bounds, if~\eqref{eq:is approximation BM invariant} were true, then Theorem~\ref{thm:l1} (more precisely, the fact that Theorem~\ref{thm:l1} holds for $\X=\ell_1^n$) would formally imply~\eqref{eq:lp variant}, and it would formally imply a version of Theorem~\ref{thm:L2 case} that is not as sharp, but it is weaker than Theorem~\ref{thm:L2 case} only in terms of lower order factors. Finally, the validity of~\eqref{eq:is approximation BM invariant} would resolve Question~\ref{Q:n^c}, and would also resolve Question~\ref{Q:q>2} up to lower order factors. To justify why this is so for Question~\ref{Q:n^c}, we use~\cite{Gia95}\footnote{In fact, the earlier bound in~\cite{ST89} on the Banach--Mazur distance to $\ell_1^n$ would suffice for our purposes. However, the first (nontrivial) bound on this quantity that was obtained~\cite{BS88} would be insufficient for our purposes.}, which proves that any $n$-dimensional normed space $(\X,\|\cdot\|_\X)$ satisfies
$$
d_{\mathrm{BM}}(\ell_1^n,\X)\lesssim n^{\frac56}.
$$
By combining this estimate with~\eqref{eq:is approximation BM invariant} and Theorem~\ref{thm:l1} (with $\X=\ell_1^n$), we would conclude that
$$
A_L(\X)\gtrsim \frac{\sqrt[6]{n}}{e^{\beta(L)\sqrt{\log n}}},
$$
for every $L\ge 1$ and $n\ge n_0(L)$, where $n_0(L)\in \N$ and $\beta(L)>0$ depend only on $L$. This would answer Question~\ref{Q:n^c} positively. Next, for Question~\ref{Q:q>2}, since for every $q\ge 2$ we have $d_{\mathrm{BM}}(\ell_q^n,\ell_1^n)\lesssim \sqrt{n}$ by~\cite{GKM66}, we would deduce from~\eqref{eq:is approximation BM invariant} in the same manner as above that for every $L\ge 1$ and $n\ge n_0(L)$,
\begin{equation}\label{eq:sqrt n for ell q lower}
A_L(\ell_q^n)\gtrsim \frac{\sqrt{n}}{e^{\beta(L)\sqrt{\log n}}}.
\end{equation}

We will next discuss issues that arise when trying to prove~\eqref{eq:is approximation BM invariant}. These issues also highlight subtleties of the almost extension problem that are not present in the Lipschitz extension problem (observe that for the latter problem, the analogue of~\eqref{eq:is approximation BM invariant} is trivially valid).

Lipschitz almost extension is a nonlinear phenomenon about the geometry of spheres. So, the appearance of  $d_{\mathrm{BM}}(\X,\Y)$ rather than $d_\Lip(S_\X,S_\Y)$ in~\eqref{eq:is approximation BM invariant} could seem unnatural. In other words, \eqref{eq:is approximation BM invariant} includes  a linearization statement by incorporating the linear quantity $d_{\mathrm{BM}}(\X,\Y)$ into a nonlinear setting. We formulated~\eqref{eq:is approximation BM invariant} as we did because such rigidity phenomena do occur quite often, and all the results on almost extension that are currently known do obey~\eqref{eq:is approximation BM invariant}. Moreover, to the best of our knowledge, evaluations of  $d_\Lip(S_\X,S_\Y)$ are not available in the literature. In our opinion, this is both a major omission and a great opportunity to study the geometry of the {\em sphere compactum}, namely the space of all the spheres of $n$-dimensional normed spaces, equipped with the multiplicative semi-metric $d_{\Lip}$. In contrast, the {\em Banach--Mazur compactum}, namely the   space of all $n$-dimensional normed spaces, equipped with the multiplicative semi-metric $d_{\mathrm{BM}}$, is a deeply studied geometric object with rich structure; see~\cite{TJ89}. The following lemma establishes rudimentary relations between $d_{\mathrm{BM}}(\X,\Y)$ and $d_{\Lip}(S_\X,S_\Y)$.

\begin{lemma}\label{lem:dist squared} Suppose that $(\X,\|\cdot\|_\X)$ and $(\Y,\|\cdot\|_\Y)$ are normed spaces with $\dim(\X)=\dim(\Y)<\infty$. Then,
\begin{equation}\label{BM Lip}
 d_\BM(\X,\Y)\lesssim d_\Lip(S_\X,S_\Y)\lesssim d_\BM(\X,\Y)^2.
\end{equation}
\end{lemma}

For ease of later reference (in the proof of Lemma~\ref{lem:dist squared} and elsewhere), we record separately the following straightforward ``normalization bound,'' which holds for any normed space $(\EE,\|\cdot\|_\EE)$.
\begin{equation}\label{eq:normalize 2}
\forall u,v\in \EE\setminus \{0\},\qquad \left\|\frac{1}{\|u\|_\EE}u-\frac{1}{\|v\|_\EE}v\right\|_\EE\le \frac{2\|u-v\|_\EE}{\max\{\|u\|_\EE,\|v\|_\EE\}}.
\end{equation}
Indeed, if $\|v\|_\EE=\max\{\|u\|_\EE,\|v\|_\EE\}$, then
\begin{multline*}
\left\|\frac{1}{\|u\|_\EE}u-\frac{1}{\|v\|_\EE}v\right\|_\EE=\left\|\left(\frac{1}{\|u\|_\EE}-\frac{1}{\|v\|_\EE}\right)u+\frac{1}{\|v\|_\EE}(u-v)\right\|_\EE\\\le \left(\frac{1}{\|u\|_\EE}-\frac{1}{\|v\|_\EE}\right)\|u\|_\EE+\frac{\|u-v\|_\EE}{\|v\|_\EE}=\frac{\|v\|_\EE-\|u\|_\EE+\|u-v\|_\EE}{\|v\|_\EE}\le \frac{2\|u-v\|_\EE}{\|v\|_\EE}=\frac{2\|u-v\|_\EE}{\max\{\|u\|_\EE,\|v\|_\EE\}}.
\end{multline*}

\begin{proof}[Proof of Lemma~\ref{lem:dist squared}] Fix $D\ge 1$ and a linear operator $T:\X\to \Y$ satisfying $\|x\|_\X\le \|Tx\|_\Y\le D\|x\|_\X$ for  $x\in \X$.  Composition with  normalization yields natural bijections  $\f:S_\X\to S_\Y$ and $\f^{-1}:S_\Y\to S_\X$ that are given by
\begin{equation}\label{eq:normalize T}
\forall x\in S_\X,\qquad \f(x)=\frac{1}{||Tx\|_\Y}Tx\qquad\mathrm{and}\qquad \forall y\in S_\Y,\qquad \f^{-1}(y)=\frac{1}{\|T^{-1}y\|_\X}T^{-1}y.
\end{equation}
Note that the fact that the two mappings in~\eqref{eq:normalize T} are indeed inverse to each other uses only the (positive) homogeneity of $T$ rather than its linearity in full. By~\eqref{eq:normalize 2} (for $\EE=\Y$), every $x,x'\in S_\X$ satisfy
\begin{equation}\label{eq:phi1}
\|\f(x)-\f(x')\|_\Y\le \frac{2\|Tx-Tx'\|_\Y}{\max\{\|Tx\|_\Y,\|Tx'\|_\Y\}}\le \frac{2D\|x-x'\|_\X}{\max\{\|x\|_\X,\|x'\|_\X\}}=2D\|x-x'\|_\X.
\end{equation}
In the same vein, by~\eqref{eq:normalize 2} (for $\EE=\X$) every $y,y'\in S_\Y$ satisfy
\begin{equation}\label{eq:phi2} \|\f^{-1}(y)-\f^{-1}(y')\|_\Y\le \frac{2\|T^{-1}y-T^{-1}y'\|_\Y}{\max\{\|T^{-1}y\|_\X,\|T^{-1}y'\|_\X\}}\le \frac{2\|y-y'\|_\Y}{\max\{\|y\|_\Y/D,\|y'\|_\Y/D\}}=2D\|y-y'\|_\Y.
\end{equation}
This shows that $d_\Lip(S_\X,S_\Y)\le 4d_\BM(\X,\Y)^2$, namely the second inequality in~\eqref{BM Lip} is proved.

It suffices to justify the first inequality in~\eqref{BM Lip} when $\|\cdot\|_\X,\|\cdot\|_\Y$ are norms on $\R^n$ that are smooth on $\R^n\setminus\{0\}$, where $n=\dim(\X)=\dim(\Y)$. Fix $D\ge 1$ and $\sigma>0$. Suppose that $\f:S_\X\to S_\Y$  is a bijection that satisfies $\sigma\|x-x'\|_\X\le \|\f(x)-\f(x')\|_\Y\le D\sigma\|x-x'\|_\X$ for all $x,x'\in S_\X$. Using Rademacher's differentiation theorem~\cite{Rad19}, there is $x\in S_\X$ at which $\f$ is differentiable. Thus, if $x^*\in S_{\X^*}$ is the supporting functional of $x$, i.e., $x^*(x)=1$, then there is a linear operator $S:\mathrm{Ker}(x^*)\to \Y$ that satisfies $\sigma \|x'\|_\X\le \|Sx'\|_\Y\le \sigma D\|x'\|_\X$ for  $x'\in \mathrm{Ker}(x^*)$. Fix $y^*\in S_{\Y^*}$ whose kernel is $S\mathrm{Ker}(x^*)\subset \Y$ and  $y\in S_\Y$ with $y^*(y)=1$. Fix also $\lambda>0$ and define $T:\X\to \Y$ by setting $Tx'=\lambda S(x'-x^*(x')x)+x^*(x')y$ for $x'\in \X$, which is well defined because $x'-x^*(x')x\in \mathrm{Ker}(x^*)$. Then, $T$ is invertible and for $y' \in \Y$  we have $T^{-1}y'=\lambda^{-1}S^{-1}(y'-y^*(y')y)+y^*(y')x$. It is mechanical to check (using  $\|S\|_{\mathrm{Ker}(x^*)\to \Y}\le \sigma D$, $\|S^{-1}\|_{S\mathrm{Ker}(x^*)\to \X}\le \sigma^{-1}$ and $\|x^*\|_{\X^*}=\|y^*\|_{\Y^*}=1$) that $\|T\|_{\X\to \Y}\le 2\lambda\sigma D+1$ and $\|T^{-1}\|_{\Y\to \X}\le 2\lambda^{-1}\sigma^{-1}+1$. So, $d_\BM(\X,\Y)\le (2\lambda\sigma D+1)(2\lambda^{-1}\sigma^{-1}+1)$. By optimizing over $\lambda>0$, we conclude that $d_\BM(\X,\Y)\le (2\sqrt{D}+1)^2$, thus proving that $ d_\BM(\X,\Y)\le 9d_\Lip(S_\X,S_\Y)$.
\end{proof}

The quadratic dependence on $d_{\mathrm{BM}}(\X,\Y)$ in~\eqref{BM Lip} is the source of the quadratic dependence on $D$ that we obtain in~\eqref{eq:reqwrite lemma with modulus}. It would be worthwhile to investigate what is the correct dependence here. Note that any power-type dependence on $d_{\mathrm{BM}}(\X,\Y)$ in~\eqref{eq:is approximation BM invariant} that is better than quadratic would imply an improvement over~\eqref{eq:A lower lq} for large enough $q$, namely if we replace~\eqref{eq:is approximation BM invariant} by  $A_L(\X)\lesssim d_{\mathrm{BM}}(\X,\Y)^{2(1-\d)} A_{\f(L)}(\Y)$ for some $\d>0$, then we would obtain the following variant of~\eqref{eq:sqrt n for ell q lower} for every $q>2$, $L\ge 1$ and $n\ge n_0(L)$.
$$
A_L(\ell_q^n)\gtrsim \frac{n^\d}{e^{\beta(L)\sqrt{\log n}}}.
$$

Among the many basic open questions about the sphere compactum, we mention the following.
\begin{itemize}
\item What is the analogue of John's theorem~\cite{Joh48} for the sphere compactum? Namely, how large can $d_{\Lip}(S^{n-1},S_\X)$ be for an $n$-dimensional normed space $(\X,\|\cdot\|_\X)$? By Lemma~\ref{lem:dist squared} (using John's theorem), the answer is bounded between universal constant multiples of $\sqrt{n}$ and $n$.

    \item What is the analogue of Gluskin's theorem~\cite{Glu81} for the sphere compactum? Namely, how large can $d_{\Lip}(S_\X,S_\Y)$ be for $n$-dimensional normed spaces $(\X,\|\cdot\|_\X)$ and $(\Y,\|\cdot\|_\Y)$? By Lemma~\ref{lem:dist squared} (using Gluskin's theorem), the answer is bounded between universal constant multiples of $n$ and $n^2$. What is the asymptotic growth rate (as $n\to \infty$) of the typical bi-Lipschitz distance between the spheres of two independently chosen Gluskin bodies?

        \item For $p,q\in [1,\infty]$, what is a the asymptotic growth rate as $n\to \infty$ of $d_{\Lip}(S_{\ell_p^n},S_{\ell_q^n})$? See~\cite{GKM66} for the asymptotic evaluation of the Banach--Mazur distance between $\ell_p^n$ and $\ell_q^n$.
    \end{itemize}

Even if we have a good upper bound on $d_{\Lip}(S_\X,S_\Y)$ at our disposal, whether in terms of $d_{\mathrm{BM}}(\X,\Y)$ or otherwise, then it is not clear how to transfer almost extension theorems (or theorems about the impossibility thereof) between $\X$ and $\Y$ (unlike classical Lipschitz extension, for which this would be immediate).

One relatively minor issue is that a bi-Lipschitz mapping need not send an $\e$-net to an $\e'$-net of its image for any $\e'>0$. This could be overcome in multiple ways, one of which is working with the notion of $(\e,\d)$-nets, in place of the more typical use of $\e$-nets. Recall the following terminology. Let $(\MM,d_\MM)$ be a metric space. For $\e,\d>0$, a subset $\NN\subset \MM$ is called an $(\e,\d)$-net of $\MM$ if it is  $\e$-dense and $\d$-separated; the former means that $\inf_{a\in \NN} d_\MM(x,a)\le \e$ for every $x\in \MM$, and the latter means that $d_\MM(a,b)\ge \d$ for every distinct $a,b\in \NN$. Thus, the usual notion of $\e$-net (which we have been using tacitly) is the same as an $(\e,\e)$-net. Fix two metric spaces $(\MM,d_\MM), (\MM',d_{\MM'})$ and $\e,\d,\alpha,\beta>0$. Suppose that $f:\MM\to \MM'$ is a mapping that satisfies the bi-Lipschitz condition $\alpha d_\MM(x,y)\le d_{\MM'}(x,y)\le \beta d_\MM(x,y)$ for all $x,y\in \MM$. Then, the image $f(\NN)$ of any $(\e,\d)$-net $\NN$ of $\MM$ is a $(\beta \e,\alpha\d)$-net of $f(\MM)$. So, a bi-Lipschitz mapping sends  an $(\e,\d)$-net to an $(\e',\d')$-net of its image for suitable parameters $\e',\d'>0$ (that are simple to control). Traditionally, almost extension theorems were proven for functions that are defined on $\e$-nets, but  obtaining (either existence or impossibility) results for $(\e,\d)$-nets rather than for $\e$-nets requires only mechanical adaptations of their proofs. We chose not to do so here in order to avoid the need to track two parameters throughout, but this could be easily taken up later if needed, by inspecting existing proofs (both for the positive extension results in the literature, and the impossibility results of the present work).

The following is a much more serious issue, which we do not know how to overcome and it may end up being an inherent obstacle to the type of transference principle for almost extension that we are hoping for. Fix  normed spaces $(\X,\|\cdot\|_\X), (\X,\|\cdot\|_{\X'}), (\Y,\|\cdot\|_\Y)$ with $\dim(\X)=\dim(\X')<\infty$. Suppose that $s>0$ and $D\ge 1$, and  $\phi:S_\X\to S_{\X'}$ is a bijection satisfying $s\|x-y\|_\X\le \|\phi(x)-\phi(y)\|_{\X'}\le Ds\|x-y\|_\X$ for $x,y\in S_\X$. If $\NN\subset S_\X$ is a net of $S_\X$ and $f:\NN\to \Y$ is $1$-Lipschitz, then $g=sf\circ \phi^{-1}:\phi(\NN)\to \Y$ is also $1$-Lipschitz.  Fix $L\ge 1$. Suppose that $G:S_{\X'}\to \Y$ is $L$-Lipschitz and denote $\alpha =\max_{u\in \phi(\NN)} \|G(u)-g(u)\|_\Y$. So,  $F=s^{-1}G\circ \phi:S_\X\to \Y$ is $(DL)$-Lipschitz and satisfies $\max_{a\in \N}\|F(a)-f(a)\|_\Y\le s^{-1}\alpha$. Thus, combining an almost extension theorem for $\X'$ with a ``vanilla'' transference of the approximating mapping $G$ back to $\X$ using the given bi-Lipschitz equivalence produces a mapping $F$ that approximates $f$ in a controlled manner, but whose Lipschitz constant is at least the distortion $D$, i.e., it  cannot be specified arbitrarily in advance per Bourgain's formulation of the almost extension problem. This encapsulates an important (and crucial for applications) aspect of the new challenge that the almost extension problem raises.   Nevertheless, the fact that~\eqref{eq:is approximation BM invariant} is consistent with the known results on almost extension could be viewed as an indication that there might a  route towards these questions that is more circuitous that the above naive reasoning.

\begin{proof}[Proof of Lemma~\ref{lem:pass between spaces}]  By Lemma~\ref{lem:dist squared} (for concreteness, we are quoting the  constants  in~\eqref{eq:phi1} and~\eqref{eq:phi2}, though for our purposes their   values do not have a substantial role), there is a bijection $\f:S_\X\to S_{\mathbf{V}}$ satisfying
\begin{equation}\label{eq:phi 2D}
\forall x,y\in S_\X,\qquad \frac{\|x-y\|_\X}{2D}\le \|\f(x)-\f(y)\|_\Y\le 2D \|x-y\|_\X.
\end{equation}
Fix $\e>0$ and an $\e$-net $\NN$ of $S_\X$. For each $x\in S_\X$  fix a net point $\nu(x)\in \NN$ satisfying $\|\nu(x)-x\|_\X\le \e$, with the convention $\nu(a)=a$ for every $a\in \NN$.  As $\f(\NN)\subset S_{\mathbf{V}}\subset S_\Y$ is $\frac{\e}{2D}$-separated by~\eqref{eq:phi 2D}, there is an $\frac{\e}{2D}$-net $\MM$ of $S_\Y$ that contains $\f(\NN)$ (simply take $\MM$ to be a maximal $\frac{\e}{2D}$-separated subset of $S_\Y$ containing $\f(\NN)$).

Suppose that $(\ZZ,\|\cdot\|_\ZZ)$ is a Banach space and that $f:S_\X\to \ZZ$ is $1$-Lipschitz. Fix an arbitrary net point $a_0\in \NN$. Define $g:\MM\to \ZZ$ by setting $g(\alpha)=f(a_0)$ whenever $x\in \MM$ satisfies $\psi(\alpha)=0$, and
\begin{equation}\label{def:g outside fiber}
\forall \alpha\in \MM\setminus \psi^{-1}(0),\qquad g(\alpha)\eqdef \frac{\|\psi(\alpha)\|_{\Y}}{18KD}\left(f\circ\nu\circ\f^{-1}\Big(\frac{\psi(\alpha)}{\|\psi(\alpha)\|_\Y}\Big)-f(a_0)\right)+f(a_0).
\end{equation}
Observe that since $\psi$ is the identity mapping on $S_{\mathbf{V}}\supseteq \f(\NN)$ and $\nu$ is the identity mapping on $\NN$, we have
\begin{equation}\label{eq:g is rescaled f}
\forall a\in \NN,\qquad g\circ\f(a)=\frac{1}{18KD}f(a).
\end{equation}
We claim that $\|g(\alpha)-g(\beta)\|_{\ZZ} \le \|\alpha-\beta\|_\Y$ for all $\alpha,\beta\in \MM$ , i.e., that $g$ is $1$-Lipschitz on $\MM$. To this end,  we may assume that  $\psi(\alpha)\neq 0$ and $\alpha\neq \beta$. The latter assumption implies that
\begin{equation}\label{eq:alpha beta far}
\|\alpha-\beta\|_\Y\ge  \frac{\e}{2D}.
\end{equation}
Consider the following identity, which holds under the convention that the terms in~\eqref{eq:galpha-gbeta2}, \eqref{eq:galpha-gbeta3} and~\eqref{eq:galpha-gbeta4} below vanish if $\psi(\beta)=0$.
\begin{align}
 g(\alpha)-g(\beta)&= \frac{\|\psi(\alpha)\|_\Y-\|\psi(\beta)\|_\Y}{18KD}\left(f\circ\nu\circ\f^{-1}\Big(\frac{\psi(\alpha)}{\|\psi(\alpha)\|_\Y}\Big)-f(a_0)\right)
\nonumber\\
&\qquad+\frac{\|\psi(\beta)\|_\Y}{18KD}\left(f\circ\f^{-1}\Big(\frac{\psi(\alpha)}{\|\psi(\alpha)\|_\Y}\Big)-
f\circ\f^{-1}\Big(\frac{\psi(\beta)}{\|\psi(\beta)\|_\Y}\Big)\right)\label{eq:galpha-gbeta2}\\
&\qquad +\frac{\|\psi(\beta)\|_\Y}{18KD}\left(f\circ\nu\circ\f^{-1}\Big(\frac{\psi(\alpha)}{\|\psi(\alpha)\|_\Y}\Big)- f\circ\f^{-1}\Big(\frac{\psi(\alpha)}{\|\psi(\alpha)\|_\Y}\Big)\right)\label{eq:galpha-gbeta3}\\
&\qquad +\frac{\|\psi(\beta)\|_\Y}{18KD}\left( f\circ\f^{-1}\Big(\frac{\psi(\beta)}{\|\psi(\beta)\|_\Y}\Big)-f\circ\nu\circ\f^{-1}\Big(\frac{\psi(\beta)}{\|\psi(\beta)\|_\Y}\Big)\right)
\label{eq:galpha-gbeta4}.
\end{align}
Since $f$ is $1$-Lipschitz,  $\psi$ is $K$-Lipschitz and $\f^{-1}$ is $(2D)$-Lipschitz, it follows that
\begin{equation}\label{eq:three tirangle}
\|g(\alpha)-g(\beta)\|_\ZZ\le  \frac{\|\alpha-\beta\|_\Y}{9D}
+\frac{\|\psi(\beta)\|_\Y}{9K}\left\| \frac{\psi(\alpha)}{\|\psi(\alpha)\|_\Y}-
\frac{\psi(\beta)}{\|\psi(\beta)\|_\Y}\right\|_\Y
+\frac{\|\psi(\beta)\|_\Y}{9KD}\sup_{x\in S_\X} \|x-\nu(x)\|_\X.
\end{equation}
Because $\psi$ is $K$-Lipschitz, using~\eqref{eq:normalize 2}  the second term in the right hand side of~\eqref{eq:three tirangle} is at most $\frac{2}{9}\|\alpha-\beta\|_\Y$. Also, since $\f(a_0)\in S_\mathbf{V}$ we have $\psi\circ \f(a_0)=\f(a_0)$, and therefore
$$\|\psi(\beta)\|_\Y\le \|\psi(\beta)-\psi\circ \f(a_0)\|_\Y+\|f(a_0)\|_\Y\le K(\|\beta-\f(a_0)\|_\Y)+1\le K(\|\beta\|_\Y+\|\f(a_0)\|_\Y)+1=2K+1.
$$
 By substituting these observations into~\eqref{eq:three tirangle} we conclude as follows that $g$ is indeed $1$-Lipschitz.
$$
\|g(\alpha)-g(\beta)\|_\ZZ\le  \frac{\|\alpha-\beta\|_\Y}{9D}
+\frac29\|\alpha-\beta\|_\Y
+\frac{2K+1}{9KD}\e\stackrel{\eqref{eq:alpha beta far}}{\le} \left(\frac{1}{9D}+\frac29+\frac{2(2K+1)}{9K}\right)\|\alpha-\beta\|_\Y\le \|\alpha-\beta\|_\Y.
$$

Denote by $\Phi:\X\to \mathbf{V}$ be the homogeneous extension of $\f$, namely $\Phi(0)=0$ and $\Phi(x)=\|x\|_\X\f(x/\|x\|_\X)$ when $x\in \X\setminus \{0\}$. Then,  because $\f$ is $(2D)$-Lipschitz and takes values in $S_{\mathbf{V}}$, by Lemma~2 in~\cite{JL84} we know that $\Phi$ is $(5D)$-Lipschitz.  Since $(L,A)$ is a $\Y$-almost extension pair and $\MM$ is an $\frac{\e}{2D}$-net of $S_\Y$, there exists an $L$-Lipschitz mapping $G:\Y\to \ZZ$ that satisfies
\begin{equation}\label{eq:G close}
\sup_{\alpha\in \MM} \|G(\alpha)-g(\alpha)\|_\ZZ\le \frac{A\e}{2D}.
\end{equation}
If we define $F\eqdef 18KDG\circ\Phi:\X\to \ZZ$, then $F$ is $90KD^2L$-Lipschitz and
\begin{equation*}
\max_{a\in \NN} \|F(a)-f(a)\|_\ZZ\stackrel{\eqref{eq:g is rescaled f}}{=}18KD \max_{a\in \NN}\|G\circ \f(a)-g\circ \f(a)\|_\ZZ\stackrel{\eqref{eq:G close}}{\le} 9KA\e.\tag*{\qedhere}
\end{equation*}
\end{proof}

\begin{remark} Extension from nets is an especially important instance of the Lipschitz extension problem, so it is worthwhile to introduce the following notation. Given a metric space $(\MM,d_\MM)$, define $\ee_{\mathsf{net}}(\MM)$ to be the supremum of $\ee(\MM,\NN)$ over all $\NN\subset \MM$ which is an $\e$-net of $\MM$ for some $\e>0$. Note that if $\X$ is a normed space, then by rescaling we see that the quantity $\ee(\X)$ is given by the case $\e=1$ of this definition, namely  $\ee(\X)$ is the supremum of $\ee(\X,\NN)$ over all $\NN\subset \X$ which is a $1$-net of $\X$. It is simple to show that
\begin{equation*}%\label{sphere ball all}
\ee_{\mathsf{net}}(S_\X)\lesssim \ee_{\mathsf{net}}(B_\X)\lesssim \ee_{\mathsf{net}}(\X).
\end{equation*}
If $\dim(\X)<\infty$, then it would also be worthwhile to study the quantity $\ee(\X,\NN)$ for any specific $1$-net $\NN$ of $\X$, because by~\cite{BK98,McM98} nets of $\X$ need not be bi-Lipschitz equivalent to each other. If $\dim(\X)=\infty$, then by~\cite{LMP00} any two nets of $\X$ are bi-Lipschitz equivalent to each other, but when the dimension is infinite extending Lipschitz functions from nets sometimes exhibits pathological behavior~\cite{Nao15}.\footnote{We checked that by a careful optimization of the proof of equation (1) in~\cite{Nao15} (which is itself inspired by~\cite{Lin64,Nao01}), one could deduce a nontrivial impossibility result for the almost extension problem whose proof is different from those in the present article as well as the proof in~\cite{Bou87}, though the resulting bound is  asymptotically weaker than what we obtain here.}

Basic questions remain open in the above context. E.g., what is the asymptotic growth rate of $\ee_{\mathsf{net}}(\ell_2^n)$? What is the asymptotic growth rate of the supremum of $\ee_{\mathsf{net}}(\X)$ over all $n$-dimensional normed spaces $\X$?

 One could also formulate the almost extension problem for an arbitrary metric space $\MM$ by asking for those $A,L\ge 1$ such that for any $\e>0$, any $\e$-net $\NN$ of $\MM$, any normed space $(\Y,\|\cdot\|_\Y)$ and any $1$-Lipschitz mapping $f:\NN\to \Y$ there is an $L$-Lipschitz mapping $F:\MM\to \Y$ that satisfies $\|F(a)-f(a)\|_\Y\le A\e$ for every net point $a\in \NN$. Investigating this question in such generality is a wide-open research direction. We warn  that when $\MM=S_\X$ for some normed space $\X$, this does  not coincide with  Bourgain's formulation, which asks for $F$ to be defined on all of $\X$ rather than just on $S_\X$. But, it is simple to check that these formulations are essentially the same if one does not mind losing universal constant factors.
 %Also, in this approximate setting it is  simple to derive bounds that are analogous to~\eqref{sphere ball all}.
\end{remark}

\section{Approximate smoothing of powers of the crinkled arc}\label{sec:l1 proof}

Here we will prove Theorem~\ref{thm:l1}. The proof is short and entirely different from the approach of~\cite{Bou87} that we recalled in Section~\ref{sec:ext intro}. We will also discuss generalizations, and variants for $\ell_p^n$ when $1<p\le 2$.

\begin{proof}[Proof of Theorem~\ref{thm:l1}]   Define $\gamma:\R\to \{0,1\}^{\R}$ as follows (the curve $\gamma$ is one possible representation of Hamlos'  ``crinkled arc,'' but it is in fact a canonical object, as shown by the uniqueness theorem of~\cite{Vit75}).
\begin{equation}\label{eq:crinkled}
\forall s\in \R,\qquad \gamma(s)\eqdef \1_{\left[\min\{s_{\phantom{j}}\!\!\!\!,0\},\max\{s,0\}\right]}.
\end{equation}
A straightforward computation shows that $\|\gamma(s)-\gamma(t)\|_{L_p(\R)}=|s-t|^{\frac{1}{p}}$ for every $p>0$ and $s,t\in \R$.

Fix $1\le q\le 2$ and $0<\e\le \frac12$ that will be determined later, and define $f=f_{q,\e}:\mathcal{N}\to \ell_q^n(L_q(\R))$ by
$$
\forall a=(a_1,\ldots,a_n)\in \mathcal{N},\qquad f(a)\eqdef \e^{1-\frac{1}{q}}\big(\gamma(a_1),\ldots,\gamma(a_n)\big).
$$
Because $\|a-b\|_{\ell_1^n}\ge \e$ for all distinct $a=(a_1,\ldots,a_n),b=(b_1,\ldots,b_n)\in \mathcal{N}$, we have
\begin{equation}\label{eq:crinckle snowflake}
\|f(a)-f(b)\|_{\ell_q^n(L_q(\R))}=\e^{1-\frac{1}{q}} \bigg(\sum_{j=1}^n \|\gamma(a_j)-\gamma(b_j)\|_{L_q(\R)}^q\bigg)^{\frac{1}{q}}=\e^{1-\frac{1}{q}}\|a-b\|_{\ell_1^n}^{\frac{1}{q}}\le \|a-b\|_{\ell_1^n}.
\end{equation}
This shows that $f$ is $1$-Lipschitz. Suppose that $F:S_{\ell_1^n}\to \ell_q^n(L_q(\R))$ is $L$-Lipschitz. We will derive~\eqref{eq:max on N} for $\Y=\ell_q^n(L_q(\R))$ and a judicious choice of $q=q(n,L)$ and  $\e=\e(n,L)$ by considering  the behavior of $F$ only on the union of $\mathcal{N}$ and the (normalized) discrete hypercube $\frac1{n}\{-1,1\}^n\subset S_{\ell_1^n}$.

By applying the Lipschitz condition along each edge of the hypercube we see that
\begin{equation}\label{eq:upper edges}
\max_{\substack{x\in \{-1,1\}^n\\j\in \n}}\Big\|F\Big(\frac{1}{n}x\Big)-F\Big(\frac{1}{n}(x_1,x_2,\ldots,x_{j-1},-x_j,x_{j+1},\ldots,x_n)\Big)\Big\|_{\ell_q^n(L_q(\R))}\le \frac{2L}{n}.
\end{equation}
Next, for each $x\in \frac{1}{n}\{-1,1\}^n\subset S_{\ell_1^n}$ fix $a_x,b_x\in \mathcal{N}$ with $\left\|\frac{1}{n}x-a_x\right\|_{\ell_1^n}\le \e$ and $\left\|-\frac{1}{n}x-b_x\right\|_{\ell_1^n}\le \e$. Then,
$$
\|a_x-b_x\|_{\ell_1^n}\ge \frac{2}{n}\|x\|_{\ell_1^n}-\Big\|\frac{1}{n}x-a_x\Big\|_{\ell_1^n}-\Big\|-\frac{1}{n}x-b_x\Big\|_{\ell_1^n}\ge 2-2\e\ge 1,
$$
where the last step holds because $\e\le\frac12$. Hence,
$$
\|f(a_x)-f(b_x)\|_{\ell_q^n(L_q(\R))}\stackrel{\eqref{eq:crinckle snowflake}}{=}\e^{1-\frac{1}{q}}\|a_x-b_x\|_{\ell_1^n}^{\frac1{q}}\ge \e^{1-\frac{1}{q}}.
$$
Also,
\begin{multline*}
\Big\|F\Big(\frac{1}{n}x\Big)-f(a_x)\Big\|_{\ell_q^n(L_q(\R))}\le \Big\|F\Big(\frac{1}{n}x\Big)-F(a_x)\Big\|_{\ell_q^n(L_q(\R))}+\|F(a_x)-f(a_x)\|_{\ell_q^n(L_q(\R))}\\
\le L\Big\|\frac{1}{n}x-a_x\Big\|_{\ell_1^n} +\max_{a\in \mathcal{N}}\|F(a)-f(a)\|_{\ell_q^n(L_q(\R))}\le L\e +\max_{a\in \mathcal{N}}\|F(a)-f(a)\|_{\ell_q^n(L_q(\R))}.
\end{multline*}
The symmetric reasoning shows that
$$
\Big\|F\Big(-\frac{1}{n}x\Big)-F(b_x)\Big\|_{\ell_q^n(L_q(\R))}\le L\e +\max_{a\in \mathcal{N}}\|F(a)-f(a)\|_{\ell_q^n(L_q(\R))}.
$$
Consequently,
\begin{align*}
\Big\|F\Big(\frac{1}{n}x\Big)&-F\Big(-\frac{1}{n}x\Big)\Big\|_{\ell_q^n(L_q(\R))}\\&\ge \|f(a_x)-f(b_x)\|_{\ell_q^n(L_q(\R))}-\Big\|F\Big(\frac{1}{n}x\Big)-f(a_x)\Big\|_{\ell_q^n(L_q(\R))}-\Big\|F\Big(-\frac{1}{n}x\Big)-
F(b_x)\Big\|_{\ell_q^n(L_q(\R))}\\
&\ge \e^{1-\frac{1}{q}}- 2L\e -2\max_{a\in \mathcal{N}}\|F(a)-f(a)\|_{\ell_q^n(L_q(\R))}.
\end{align*}
In other words,
\begin{equation}\label{eq:lower diagonals}
\min_{x\in \{-1,1\}^n} \Big\|F\Big(\frac{1}{n}x\Big)-F\Big(-\frac{1}{n}x\Big)\Big\|_{\ell_q^n(L_q(\R))}\ge \e^{1-\frac{1}{q}}- 2L\e -2\max_{a\in \mathcal{N}}\|F(a)-f(a)\|_{\ell_q^n(L_q(\R))}.
\end{equation}

A fundamental inequality of Enflo~\cite{Enf69} asserts that any function $\varphi:\{-1,1\}^n\to \ell_q^n(L_q(\R))$ satisfies
\begin{equation}\label{eq:min max}
\min_{x\in \{-1,1\}^n} \|\varphi(x)-\varphi(-x)\|_{\ell_q^n(L_q(\R))}\le n^{\frac{1}{q}} \max_{\substack{x\in \{-1,1\}^n\\j\in \n}}\|\varphi(x)-\varphi(x_1,x_2,\ldots,x_{j-1},-x_j,x_{j+1},\ldots,x_n)\|_{\ell_q^n(L_q(\R))}.
\end{equation}
In combination with~\eqref{eq:upper edges} and~\eqref{eq:lower diagonals}, we therefore conclude that
\begin{equation}\label{eq:sup lower 1}
\max_{a\in \mathcal{N}}\|F(a)-f(a)\|_{\ell_q^n(L_q(\R))}\ge \frac12 \e^{1-\frac{1}{q}}-L\e-Ln^{\frac{1}{q}-1}= \Big((2L)^{-1}\e^{-\frac{1}{q}}-\e^{-1}n^{\frac{1}{q}-1}-1\Big)L\e.
\end{equation}
Hence, if we denote
$$
M(n,L)\eqdef \max_{\substack{0<\e\le \frac12\\ 1\le q\le 2}}\Big((2L)^{-1}\e^{-\frac{1}{q}}-\e^{-1}n^{\frac{1}{q}-1}\Big),
$$
and choose $0<\e_\opt=\e_\opt(n,L)\le \frac12$ and $1\le q_\opt=q_\opt(n,L)\le 2$ such that
\begin{equation}\label{eq:maximizers}
(2L)^{-1}\e_\opt^{-\frac{1}{q_\opt}}-\e_\opt^{-1}n^{\frac{1}{q_\opt}-1}=M(n,L),
\end{equation}
then~\eqref{eq:sup lower 1} implies that
\begin{equation}\label{eq:MnL version}
\max_{a\in \mathcal{N}}\|F(a)-f(a)\|_{\ell_{q_\opt}^n(L_{q_\opt}(\R))}\ge \big(M(n,L)-1\big)L\e_\opt.
\end{equation}

It remains to explain why~\eqref{eq:MnL version} implies~\eqref{eq:max on N}; in fact~\eqref{eq:MnL version} is asymptotically stronger than~\eqref{eq:max on N} in terms of the lower order factor. The values of $0<\e_\opt\le \frac12$ and $1\le q_\opt\le 2$  for which~\eqref{eq:maximizers} holds are determined by
\begin{equation}\label{eq:two equation}
n^{\left(1-\frac{1}{q_\opt}\right)^2}=2q_\opt L\qquad\mathrm{and}\qquad \e=\Big(\frac{1}{n}\Big)^{\frac{1}{q_\opt}}.
\end{equation}
This assertion includes the  claim that, thanks to the assumption $5L\le\sqrt[4]{n}$, there is a unique $1\le q_{\opt}\le 2$ for which the first equation in~\eqref{eq:two equation} holds. The justification that the maximizers in~\eqref{eq:maximizers} are characterized by~\eqref{eq:two equation} consists of a (somewhat tedious) computation using elementary calculus that we omit  because we will derive~\eqref{eq:max on N} by taking the following  suboptimal setting of parameters in~\eqref{eq:sup lower 1}; we prefer this presentation despite the fact that it sacrifices a lower order improvement of~\eqref{eq:MnL version}, because it    is shorter.
\begin{equation}\label{eq:our choices sub optimal}
\e_*=\e_*(n,L)\eqdef \frac{e^{\sqrt{(\log n)\log(4L)}}}{n} \qquad\mathrm{and}\qquad q_*=q_*(n,L)\eqdef \frac{1}{1-\frac{\sqrt{\log (4L)}}{\sqrt{\log n}}}.
\end{equation}
Note that we indeed have $\e_*\le \frac12$ and $1<q_*\le 2$ because of the assumption $5L\le\sqrt[4]{n}$. Hence,
\begin{equation*}
\max_{a\in \mathcal{N}}\|F(a)-f(a)\|_{\ell_{q_*}^n(L_{q_*}(\R))}\stackrel{\eqref{eq:sup lower 1}}{\ge} \frac12 \e_*^{1-\frac{1}{q_*}}-L\e_*-Ln^{\frac{1}{q_*}-1}\stackrel{\eqref{eq:our choices sub optimal}}{=}\Big(ne^{-2\sqrt{(\log n)\log(4L)}}-1\Big)L\e_*.
\end{equation*}
This implies the desired estimate~\eqref{eq:max on N} for $\e=\e_*$ because  $5L\le\sqrt[4]{n}$.
\end{proof}

\begin{remark} The above proof of Theorem~\ref{thm:l1} does not use fine properties of the crinkled ark. It only needs a mapping from $\ell_1^n$ to an $L_q(\mu)$ space that is $\frac{1}{q}$-H\"older with constant $C_1$ on $S_{\ell_1^n}$, and which sends any pair of points in $S_{\ell_1^n}$ that are at distance at least $C_2$ to points in $L_q(\mu)$  that are at distance at least $C_3$, where $C_1,C_2,C_3>0$ are constants independent of $n\in \N$ and $1<q<2$ whose only impact is on the resulting lower order factor. Similar statements could therefore be derived for other mappings, such as the Mazur map~\cite{Maz29}. However,  it would be independently interesting to find  the best behavior in the approximation lower bound~\eqref{eq:sup lower 1} when using the crinkled ark: Is the lower order factor that this  induces in~\eqref{eq:max on N} necessary for the specific example of (powers of) the crinkled ark the we consider here?
\end{remark}

\begin{remark} In the proof of Theorem~\ref{thm:l1}, the role of $L_q(\mu)$ as the target space is only through the validity of~\eqref{eq:min max}. For this,  it suffices for the target to have Enflo type $q$ (using terminology of~\cite{BMW86}), and by the  remarkable work~\cite{IvHV20} this is equivalent to the target having type $q$ in the linear sense (recall~\eqref{eq:type cotype def}).

A quite substantial amount of bi-Lipschitz invariants is available  as an outgrowth of the Ribe program (see~\cite{Nao12,Bal13}, though by now these surveys are out of date and more such invariants are known). Due to the inherent quantitative nature of these invariants, they could in principle be used to prove the kind of non-extension results that are the topic of the present work. Indeed, the above reasoning shows how  nonlinear type is useful for this purpose (applications of nonlinear notions of type to impossibility of Lipschitz extension were found in other contexts in~\cite[page~137]{JLS86} and~\cite[Section~8]{NP11}). In the course of the research that led to the present work, we systematically examined the relevance of all of the aforementioned bi-Lipschitz invariants (of which we are aware) for the purpose of proving impossibility of either Lipschitz extension or  almost extension. This study is too lengthy to include here. It suffices to say that while some of these invariants turn out to be applicable in our context, we only succeeded to use them to derive bounds that are inferior to those that we present here. For example, we could use metric $X_p$ inequalities~\cite{NS16-Xp,Nao16-riesz} to  prove a variant of~\eqref{eq:A lower lq} which is asymptotically weaker in terms of lower-order factors (as $n\to \infty$) via a route that is completely different from how~\eqref{eq:A lower lq} is proven here. At the same time, the use of many of those invariants (such as metric cotype~\cite{MN08}) to prove (almost) non-extendability eludes us, and it would be worthwhile to obtain a principled understanding of the reason why this is so. In some cases, we do see  conceptual reasons for such difficulties. As an example, consider the metric variant of the expander Poincar\'e inequality~\cite{LLR95,Mat97,MN14}, which is commonly used to remove lower-order factors (as done specifically  for Lipschitz extension in~\cite{NR17}), so a priori one could hope to use it to sharpen Theorem~\ref{thm:l1}. However, it is impossible to realize expanders faithfully as subsets of a low-dimensional normed space; certainly not bi-Lipschitzly~\cite{Nao17}, but even weaker low-dimensional representations of expanders are impossible~\cite{Nao18,Nao19}. This appears to be an inherent obstacle to proving lower bounds on (almost) Lipschitz extension from subsets of a finite dimensional normed space that grow like a power of its dimension rather than logarithmically.
\end{remark}

\subsection{Approximate smoothing of powers of the helix}\label{sec:lp variant} We stated and proved Theorem~\ref{thm:l1} separately because this is how we show that the original (full-generality) formulation of Bourgain's almost extension theorem is optimal up to lower order factors, and also because the crinkled ark is a canonical object. But, the reasoning of the proof of Theorem~\ref{thm:l1}  can be readily adjusted to yield almost optimal (up to lower order factors) variants when the source space is $\ell_p^n$ for the range $1\le p< 2$; the only difference is replacing the crinkled ark by a $\theta$-helix. Our proof breaks down when $p=2$. Of course, Theorem~\ref{thm:L2 case} treats   the endpoint case $p=2$ even more satisfactorily (sharp up to constant factors), but we will prove Theorem~\ref{thm:L2 case} (in Section~\ref{eq:l2 proof}) via a completely different and more involved reasoning that relies on Euclidean symmetries.

Our next goal  is to briefly explain how to adapt the proof of Theorem~\ref{thm:l1} so as to obtain its generalization that we already stated in~\eqref{eq:lp variant}. Specifically, we will  show that if $A,L\ge 1$, $n\in \N$ and $1\le p< 2$ are such that
\begin{equation}\label{eq:lower n p-2}
n\ge (5L)^{\frac{4p}{(2-p)^2}},
\end{equation}
and $(L,A)$ is an $\ell_p^n$-almost extension pair, then
\begin{equation}\label{eq:repeat A lower p}
A\gtrsim Ln^{\frac{1}{p}}e^{-2\sqrt{\frac{\log (4L)}{p}}\log n}.
\end{equation}
This  is sharp up to lower order factors (as $n\to \infty$) because it is proved in~\cite{NS18-extension} that for some universal constant $C>0$, if $L>C$ and $n\in \N$, then $(L,A)$ is an $\ell_p^n$-almost extension pair, where $A\ge 1$ satisfies
$$
A\lesssim \frac{n^{\frac{1}{p}}}{L}.
$$

By~\cite[Corollary~1]{Sch38} for each $\theta\in (0,1)$ there is a curve $h_\theta:\R\to \ell_2$, called the $\theta$-helix, that satisfies
\begin{equation}\label{eq:theta helix property}
\forall s,t\in \R,\qquad \|h_\theta(s)-h_\theta(t)\|_{\ell_2}=|s-t|^\theta.
\end{equation}
\begin{comment}
See~\cite{Ass83,Kah81,Tal92,GK15} for approximate realizations of this curve that yield low-dimensional variants of the mappings that we study below; we omit this discussion since it is not used in the present context, though we will need such variants in forthcoming related work.
\end{comment}
For $n\in \N$ and $x=(x_1,\ldots,x_n)\in \R^n$ write $h_\theta^{\otimes n}(x)=(h_\theta(x_1),\ldots,h_\theta(x_n))\in (\ell_2)^n$. By applying~\eqref{eq:theta helix property} coordinate-wise we see that for every $0<p\le q<\infty$ we have
$$
\forall x,y\in \R^n,\qquad \big\|h_{\frac{p}{q}}^{\otimes n}(x)-h_{\frac{p}{q}}^{\otimes n}(y)\big\|_{\ell_q^n(\ell_2)}=\|x-y\|_{\ell_p^n}^{\frac{p}{q}}.
$$

Fix $0<\e\le \frac12$ and let $\NN$ be an $\e$-net of $S_{\ell_p^n}$. For  $1\le p<q\le 2$, define $f=f_{p,q,\e,n}:\NN\to \ell_q^n(\ell_2)$ by
$$
f\eqdef \e^{1-\frac{p}{q}}h_{\frac{p}{q}}^{\otimes n}.
$$
Then, $f$ is $1$-Lipschitz by reasoning mutatis mutandis as in~\eqref{eq:crinckle snowflake}. Suppose that $L,A\ge 1$ and $F:\ell_p^n\to \ell_q^n(\ell_2)$ is $L$-Lipschitz and satisfies $\|F(a)-f(a)\|_{\ell_q^n(\ell_2)}\le A\e$ for every $a\in \NN$. Our goal is proving that $A$ obeys~\eqref{eq:repeat A lower p}.

Since $F$ is $L$-Lipschitz we have analogously to~\eqref{eq:upper edges},
\begin{equation}\label{eq:upper edges helix}
\max_{\substack{x\in \{-1,1\}^n\\j\in \n}}\Big\|F\big(n^{-\frac{1}{p}}x\big)-F\big(n^{-\frac{1}{p}}(x_1,x_2,\ldots,x_{j-1},-x_j,x_{j+1},\ldots,x_n)\big)\Big\|_{\ell_q^n(\ell_2)}\le 2Ln^{-\frac{1}{p}}.
\end{equation}
The same reasoning that led to~\eqref{eq:lower diagonals} leads mutatis mutandis to the following estimate
\begin{equation}\label{eq:lower diagonals helix}
\min_{x\in \{-1,1\}^n} \Big\|F\big(n^{-\frac{1}{p}}x\big)-F\big(-n^{-\frac{1}{p}} x\big)\Big\|_{\ell_q^n(\ell_2)}\ge \e^{1-\frac{p}{q}}- 2L\e -2A\e.
\end{equation}
A substitution of~\eqref{eq:upper edges helix} and~\eqref{eq:lower diagonals helix} into~\eqref{eq:min max}, which is valid because $\ell_2$ embeds isometrically into $L_q(\R)$ (see e.g.~\cite[Chapter~III.A]{Woj91}), gives the following  bound on $A$, which holds when $0<\e\le \frac12$ and $p\le q\le 2$.
\begin{equation}\label{eq:A lower p version}
A\ge  \Big((2L)^{-1}\e^{-\frac{p}{q}}-\e^{-1}n^{\frac{1}{q}-\frac{1}{p}}-1\Big)L.
\end{equation}
We arrive at~\eqref{eq:repeat A lower p} by optimizing similarly to the conclusion of the proof of Theorem~\ref{thm:l1}. Specifically, due to~\eqref{eq:lower n p-2} we see that~\eqref{eq:repeat A lower p} is a consequence of~\eqref{eq:A lower p version} by substituting the following values of $q$ and $\e$ into~\eqref{eq:A lower p version}, while noting that indeed $0<\e\le \frac12$ and $p\le q\le 2$ thanks to~\eqref{eq:lower n p-2}.
\begin{equation*}
\e=\frac{1}{n}e^{\sqrt{\frac{\log (4L)}{p}\log n}}\qquad\mathrm{and}\qquad q=\frac{p}{1-\frac{\sqrt{p\log(4L)}}{\sqrt{\log n}}}.\tag*{\qed}
\end{equation*}

\section{Approximate stochastic retraction from the Euclidean sphere onto its net}\label{eq:l2 proof}

Our goal here is to prove Theorem~\ref{thm:L2 case} by reasoning via the dual characterization in part~${(\it 2)}$ of Proposition~\ref{prop:duality}.  Prior to doing so, we will make some preparatory comments that will be used in the proof.

Denote the standard scalar product on $\R^n$ by $\langle \cdot,\cdot\rangle :\R^n\times \R^n\to \R$, namely $\langle x,y\rangle =x_1y_1+\ldots+x_ny_n$ for  $x=(x_1,\ldots,x_n),y=(y_1,\ldots,y_n)\in \R^n$. For $d>0$, the $d$-dimensional Hausdorff measure that is induced on $\R^n$ by the $\ell_2^n$ metric will be denoted $\cH^d$. So, for any $k\in \{1,\ldots,n-1\}$ and $\rho>0$,
\begin{equation}\label{eq:sphere volume formula}
\cH^k\big(\rho S^k\big)=\frac{2\pi^{\frac{k}{2}}}{\Gamma\big(\frac{k}{2}\big)}\rho^k,
\end{equation}
where we use the natural identification ${S}^k=\{(x_1,\ldots,x_{k+1},0,\ldots,0)\in \R^n: x_1^2+\ldots+x_{k+1}^2=1\}\subset \R^n$.

For a direction $z\in {S}^{n-1}$, denote its orthogonal complement by $z^\perp=\{w\in \R^n:\ \langle w,z\rangle=0\}$, and its positive and negative half spaces by $z_+^\perp=\{w\in \R^n:\ \langle w,z\rangle>0\}$ and $z_-^\perp=\{w\in \R^n:\ \langle w,z\rangle<0\}$, respectively.

Let $\mu$ be a (signed) Borel measure of finite total variation on $\R^{n}$. Define the  {\em imbalance} of $\mu$  to be the function $\imb_\mu:{S}^{n-1}\to \R$ that associates to a direction $z\in {S}^{n-1}$ the (potentially negative) amount by which the $\mu$--mass  of the positive half space $z^\perp_+$ exceeds the $\mu$--mass  of the negative half space $z^\perp_-$, i.e.,

\begin{equation}\label{eq:def imbalance}
\imb_\mu(z)\eqdef \mu\big(z^\perp_+\big)-\mu\big(z^\perp_-\big)=\int_{\R^n} \sign\big(\langle z,a\rangle\big) \ud\mu(a),
\end{equation}
where, for the second equality in~\eqref{eq:def imbalance}, and throughout what follows, we use the convention $\sign(0)=0$.

Let $\cL\subset L_2(\cH^{n-1}\lfloor_{{S}^{n-1}})$ be the subspace of linear functions in $L_2(\cH^{n-1}\lfloor_{{S}^{n-1}})$, where  $\cH^{n-1}\lfloor_{{S}^{n-1}}$ denotes the restriction of $\cH^{n-1}$ to ${S}^{n-1}$, namely it is the surface area measure on ${S}^{n-1}$.  If the orthogonal projection from $L_2(\cH^{n-1}\lfloor_{{S}^{n-1}})$ onto $\cL$ is denoted by $\proj_\cL$, then  the imbalance of any signed Borel measure $\mu$ of finite total variation that is supported on ${S}^{n-1}$ satisfies the following identity for every $x\in {S}^{n-1}$.
\begin{equation}\label{eq:projection identity}
 \proj_\cL\big(\imb_\mu\big)(x)=\frac{2\cH^{n-2}({S}^{n-2})}{\cH^{n-1}({S}^{n-1})}\sum_{j=1}^n x_j\int_{{S}^{n-1}}a_j\ud\mu(a)=\frac{2\cH^{n-2}({S}^{n-2})}{\cH^{n-1}({S}^{n-1})}\left\langle x, \text{\tt Moment}_\mu\right\rangle,
\end{equation}
where we recall the notation~\eqref{eq:moment def}, namely the vector $\text{\tt Moment}_\mu=\int_{S^{n-1}} a\ud\mu(a)\in \R^n$ is the moment of $\mu$.

Since the (normalized) coordinate functions $z\mapsto z_1\sqrt{n/\cH^{n-1}({S}^{n-1})},\ldots,z\mapsto z_n\sqrt{n/\cH^{n-1}({S}^{n-1})}$ form an orthonormal basis of $\cL$, the verification of~\eqref{eq:projection identity} consists of the following straightforward computation for any $j\in \n$, which we include here for the sake of completeness.
\begin{align*}
\int_{{S}^{n-1}}z_j\imb_\mu(z)\ud\cH^{n-1}(z)&=\int_{{S}^{n-1}}z_j \bigg(\int_{{S}^{n-1}}\sign\big(\langle z,a\rangle\big) \ud\mu(a)\bigg)\ud\cH^{n-1}(z)\\&=\int_{{S}^{n-1}}\bigg(\int_{{S}^{n-1}}z_j\sign\big(\langle z,a\rangle\big) \ud\cH^{n-1}(z)\bigg)\ud\mu(a)\\
&=\int_{{S}^{n-1}}\bigg(\int_{-1}^1\int_{(a^\perp+ta)\cap {S}^{n-1}}z_j\sign(t) \ud\cH^{n-2}(z)\ud t\bigg)\ud\mu(a)\\
&=\int_{{S}^{n-1}}\bigg(\int_{-1}^1\sign(t)\int_{\sqrt{1-t^2}{S}^{n-2}}(w_j+ta_j) \ud\cH^{n-2}(w)\ud t\bigg)\ud\mu(a)\\
&=\int_{{S}^{n-1}}\bigg(\int_{-1}^1(1-t^2)^{\frac{n-2}{2}}|t|\cH^{n-2}\big({S}^{n-2}\big)a_j\ud t\bigg)\ud\mu(a)\\
&=\frac{2}{n}\cH^{n-2}\big({S}^{n-2}\big)\int_{{S}^{n-1}}a_j\ud\mu(a).
\end{align*}

Note that~\eqref{eq:projection identity} implies that any Borel measure $\mu$ on $S^{n-1}$ of finite total variation satisfies the following sharp\footnote{Equality holds when $\imb_\mu\in \mathcal{L}$, a sufficient condition for which is that $\mu$ is absolutely continuous with respect to $\cH^{n-1}\lfloor_{{S}^{n-1}}$ and its density belongs to $\mathcal{L}$; we did not attempt to check whether this condition is also necessary.} lower bound on the $L_2$ norm of its imbalance in terms of the Euclidean length of its moment.
\begin{equation}\label{eq:imbalance moment inequality}
\left\|\imb_\mu\right\|_{L_2(\cH^{n-1}\lfloor_{{S}^{n-1}})}\ge \frac{2\cH^{n-2}({S}^{n-2})}{\sqrt{n\cH^{n-1}({S}^{n-1})}}\left\|\text{\tt Moment}_\mu\right\|_{\ell_2^n}.
\end{equation}

\begin{proof}[Proof of Theorem~\ref{thm:L2 case}] Fix $\e>0$ whose value will be specified later so as to optimize the ensuing reasoning. Let $\NN$ be an $\e$-net of $S^{n-1}$. Suppose that $\{\mu_x\}_{x\in S^{n-1}}$ is a collection of elements of the Wasserstein-1 space $\W_1(\NN)$, i.e., each $x\in S^{n-1}$ is associated to a signed measure $\mu_x$ that is supported on $\NN$ whose total mass satisfies $\mu_x(\NN)=0$. By Proposition~\ref{prop:duality}, assuming that
\begin{equation}\label{eq:Lip L}
\forall x,y\in S^{n-1},\qquad \|\mu_x-\mu_y\|_{\W_1(\mathcal{N})}\le L\|x-y\|_{\ell_2^n},
\end{equation}
the goal of Theorem~\ref{thm:L2 case} is to demonstrate that for a judicious choice of $\e>0$ the Lipschitz condition~\eqref{eq:Lip L} entails that  any fixed probability measure $\nu$ on $\NN$ satisfies
\begin{equation}\label{eq:goal within}
A_\nu\eqdef \frac{1}{\e}\max_{a\in \NN} \|\mu_a-\bd_a+\nu\|_{\W_1(\NN)}\gtrsim \frac{\sqrt{n}}{L}.
\end{equation}

To prove~\eqref{eq:goal within},   for each direction $z\in S^{n-1}$ define $\psi_z:{S}^{n-1}\to \R$ by setting for every $x\in {S}^{n-1}$,
\begin{equation}\label{eq:our psi def}
\psi_z(x)\eqdef \imb_{\mu_x-\mu_{-x}}(z)\stackrel{\eqref{eq:def imbalance}}{=}\mu_x\big(z^\perp_+\big)-\mu_{-x}\big(z^\perp_+\big)+\mu_{-x}\big(z^\perp_-\big)-\mu_{x}\big(z^\perp_-\big).
\end{equation}
We will next deduce from~\eqref{eq:Lip L} that
\begin{equation}\label{eq:psi Lip goal}
\sup_{z\in S^{n-1}}\|\psi_z\|_{\Lip}\le \frac{4L}{\e}.
\end{equation}
Indeed, for every fixed direction $z\in S^{n-1}$ and distinct net points $a,b\in \mathcal{N}$ we have
$$
\big|\sign\big(\langle z,a\rangle\big)-\sign\big(\langle z,b\rangle\big)\big|\le 2 \le \frac{2}{\e}\|a-b\|_{\ell_2^n},
$$
because $\|a-b\|_{\ell_2^n}\ge \e$. So, the mapping $a \mapsto \sign(\langle z,a\rangle)$ is $\frac{2}{\e}$-Lipschitz on $\mathcal{N}$. Hence, for every $x,y\in S^{n-1}$,
\begin{multline*}
|\psi_z(x)-\psi_z(y)|\stackrel{\eqref{eq:def imbalance}}{=}\bigg|\int_{\mathcal{N}} \sign\big(\langle z,a\rangle\big)\ud (\mu_x-\mu_{-x}-\mu_y+\mu_{-y})(a)\bigg|\stackrel{\eqref{eq:KR}}{\le} \frac{2}{\e}\|\mu_x-\mu_{-x}-\mu_y+\mu_{-y}\|_{\W_1(\mathcal{N})}\\
\le \frac{2}{\e}\left(\|\mu_x-\mu_y\|_{\W_1(\mathcal{N})}+\|\mu_{-x}-\mu_{-y}\|_{\W_1(\mathcal{N})}\right)\stackrel{\eqref{eq:Lip L}}{\le} \frac{4L}{\e}\|x-y\|_{\ell_2^n}.
\end{multline*}
Thus we indeed have $\|\psi_z\|_{\Lip}\le \frac{4L}{\e}$ for every fixed direction $z\in {S}^{n-1}$.

Fixing $z\in {S}^{n-1}$, note that by the definition~\eqref{eq:our psi def}  $\psi_z$ is odd, and hence  $\int_{{S}^{n-1}}\psi_z\ud\cH^{n-1}=0$. Because the spectral gap of the Laplace--Beltrami operator on ${S}^{n-1}$ equals $\frac{1}{n-1}$ (see e.g.~\cite{Cha84}), we therefore have
\begin{equation}\label{eq:to integrate for getting imbablance}
\int_{{S}^{n-1}}\psi_z^2\ud\cH^{n-1}\le \frac{1}{n-1}\int_{{S}^{n-1}}\|\nabla_{{S}^{n-1}}\psi_z\|_{\ell_2^n}^2\ud\cH^{n-1}\le\frac{\cH^{n-1}({S}^{n-1})}{n-1}\|\psi_z\|_{\Lip}^2\stackrel{\eqref{eq:psi Lip goal}} {\le}  \frac{32L^2\cH^{n-1}({S}^{n-1})}{n\e^2}.
\end{equation}
By integrating~\eqref{eq:to integrate for getting imbablance} over $z\in {S}^{n-1}$ and interchanging the order of integration, we deduce that
\begin{equation}\label{eq:upper imbalance integral}
\int_{{S}^{n-1}} \big\|\imb_{\mu_x-\mu_{-x}}\big\|_{L_2(\cH^{n-1}\lfloor_{{S}^{n-1}})}^2\ud \cH^{n-1}(x)\le  \frac{32L^2\cH^{n-1}({S}^{n-1})^2}{n\e^2}.
\end{equation}
By combining~\eqref{eq:upper imbalance integral} with the Imbalance/Moment inequality~\eqref{eq:imbalance moment inequality}, we therefore have
\begin{equation}\label{eq:moment upper bound}
\int_{S^{n-1}}\big\|\text{\tt Moment}_{\mu_x}-\text{\tt Moment}_{\mu_{-x}}\big\|_{\ell_2^n}^2\ud \cH^{n-1}(x)\le \frac{8L^2\cH^{n-1}({S}^{n-1})^3}{\e^2\cH^{n-2}({S}^{n-2})^2}.
\end{equation}
As the left hand side of~\eqref{eq:moment upper bound} is at least $\cH^{n-1}({S}^{n-1})\inf_{x\in S^{n-1}} \big\|\text{\tt Moment}_{\mu_x}-\text{\tt Moment}_{\mu_{-x}}\big\|_{\ell_2^n}^2$, we conclude that
\begin{equation}\label{eq:pass to inf}
\inf_{x\in S^{n-1}} \big\|\text{\tt Moment}_{\mu_x}-\text{\tt Moment}_{\mu_{-x}}\big\|_{\ell_2^n}\le \frac{\sqrt{8}L\cH^{n-1}({S}^{n-1})}{\e\cH^{n-2}({S}^{n-2})}\le \frac{16L}{\e\sqrt{n}},
\end{equation}
where the last step of~\eqref{eq:pass to inf} follows from Stirling's formula using~\eqref{eq:sphere volume formula}.

For each $z\in S^{n-1}$ fix an arbitrary net point $a(z)\in \NN$ for which $\|z-a(z)\|_{\ell_2^n}\le \e$. Then, for any $x,y\in S^{n-1}$,
\begin{eqnarray*}
\|x-y\|_{\ell_2^n}&\le& \|a(x)-a(y)\|_{\ell_2^n}+\|x-a(x)\|_{\ell_2^n}+\|y-a(y)\|_{\ell_2^n}\\
&\le& \|a(x)-a(y)\|_{\ell_2^n}+2\e\\
&=& \big\|\text{\tt Moment}_{\bd_{a(x)}}-\text{\tt Moment}_{\bd_{a(x)}}\big\|_{\ell_2^n}+2\e\\
&\le& \big\|\text{\tt Moment}_{\mu_x}-\text{\tt Moment}_{\mu_{y}}\big\|_{\ell_2^n}+\big\|\text{\tt Moment}_{\mu_x-\mu_{a(x)}}\big\|_{\ell_2^n}+\big\|\text{\tt Moment}_{\mu_y-\mu_{a(y)}}\big\|_{\ell_2^n}
\\&&\qquad +\big\|\text{\tt Moment}_{\mu_{a(x)}-\bd_{a(x)}+\nu}-\text{\tt Moment}_{\mu_{a(y)}-\bd_{a(y)}+\nu}\big\|_{\ell_2^n}+2\e\\
&\stackrel{\eqref{eq:moment wasserstein}}{\le} & \big\|\text{\tt Moment}_{\mu_x}-\text{\tt Moment}_{\mu_{y}}\big\|_{\ell_2^n}+\|\mu_x-\mu_{a(x)}\|_{\W_1(\NN)}+\|\mu_y-\mu_{a(y)}\|_{\W_1(\NN)}
\\&&\qquad + \|\mu_{a(x)}-\bd_{a(x)}+\nu\|_{\W_1(\NN)}+\|\mu_{a(y)}-\bd_{a(y)}+\nu\|_{\W_1(\NN)}+2\e\\
&\stackrel{\eqref{eq:Lip L}\wedge \eqref{eq:goal within}}{\le}& \big\|\text{\tt Moment}_{\mu_x}-\text{\tt Moment}_{\mu_{y}}\big\|_{\ell_2^n}+2L\e+2A_\nu\e+2\e
\end{eqnarray*}
By contrasting the case $y=-x$ with~\eqref{eq:pass to inf} we conclude that
\begin{equation}\label{eq:to optimize epsilon quadratic}
2\le \frac{16L}{\e\sqrt{n}}+2(L+A_\nu+1)\e.
\end{equation}
The value of $\e$ that minimizes the right hand side of~\eqref{eq:to optimize epsilon quadratic} is $
\e_{\min}^\nu\eqdef \frac{\sqrt{8L}}{\sqrt{L+A_\nu+1}}\cdot \frac{1}{\sqrt[4]{n}}
$, using which~\eqref{eq:to optimize epsilon quadratic} gives
$$
A_\nu\ge \frac{\sqrt{n}}{32L}-L-1.
$$
This implies the desired lower bound~\eqref{eq:goal within} because of our assumption $6L\le\sqrt[4]{n}$.
\end{proof}

\begin{remark}
As an alternative route to Theorem~\ref{thm:L2 case}, we checked that by implementing a suitable adaptation of  the strategy of the proof of~\cite[Theorem~1.17]{MN13} (which requires, in particular, a  setting of parameters that differs from what was needed in~\cite{MN13}), one can obtain a different derivation of the sharp lower bound~\eqref{eq:max on N hilbert}. Even though this approach results in a more complicated example, it has the advantage that the target space $\Y$ has Banach--Mazur distance $O(1)$ to some  subspace of $\ell_1$, and therefore, in particular, it has cotype $2$. In contrast, using~\cite{NS07} one can show that the Banach--Mazur distance between $\W_1(\NN)$, which is the target  space that we use in the above proof of Theorem~\ref{thm:L2 case}, and any subspace of $\ell_1$ is at least a universal constant multiple of $\sqrt{\log n}$. It is unknown whether $\W_1(\NN)$, or even $\W_1(\ell_2)$, has finite cotype (see~\cite[Question~7]{Nao18}); this is related to the old question~\cite{Bou86} whether $\W_1(\R^2)$  has finite cotype.
\end{remark}

%\cite{Bou87,Nao15,Nao01,GNS12,NS16,NR15,Beg99,Lin64,JL84,JLS86,Nao10,BB12,LN05,Nao14,Maz29,BB06,MN04,BDK65,WW75,vNS41,
%Sch38,NS16-Xp,Ban32,MN13-bary,Nao16,ANN11,Lov88,BMW86,NS02,Pis86,Wil35,AB15,Kah81,Tal92,GK15,Kal04,Kal12} \cite[Chapter~9]{BL00}
%\cite[Section~7]{Sch38}

\bibliographystyle{alphaabbrvprelim}
\bibliography{almost-ext}

 \end{document}